\documentclass[12pt]{amsart}
\usepackage[english]{babel}
\usepackage[utf8]{inputenc}
\usepackage{fullpage}
\usepackage{mathtools}
\usepackage{amsmath}
\usepackage{tikz-cd}

\usepackage[colorlinks]{hyperref}
\renewcommand\eqref[1]{(\ref{#1})} %Need with hyperref

\newtheorem{defi}{Definition}[section]
\newtheorem{exa}[defi]{Example}
\newtheorem{lema}[defi]{Lemma}
\newtheorem{teo}[defi]{Theorem}
\newtheorem{rem}[defi]{Remark}
\newtheorem*{asu}{Assumption}
\newtheorem{coro}[defi]{Corollary}
\newtheorem{pro}[defi]{Proposition}
\newtheorem*{rem*}{Remark}

\newcommand{\Ombar}{\overline{ \Omega}}
\newcommand{\interior}[1]{%
  {\kern0pt#1}^{\mathrm{o}}%
}
\newcommand{\MO}{\mathfrak{L}}
\newcommand{\I}{\mathcal{I}}
\newcommand{\T}{\mathbb{T}}
\newcommand{\Td}{\mathbb{T}^d}
\newcommand{\C}{\mathbb{C}}

\newcommand{\R}{\mathbb{R}}
\newcommand{\N}{\mathbb{N}}
\newcommand{\Z}{\mathbb{Z}}
\newcommand{\norm}[1]{\left\lVert#1\right\rVert}
\newcommand{\esp}{\text{  }}

\begin{document}

\title[NONHARMONIC ANALYSIS, SOME SPECTRAL PROPERTIES]
 {NONHARMONIC GOHBERG'S LEMMA, GERSHGORIN THEORY AND HEAT EQUATION ON MANIFOLDS WITH BOUNDARY}

%----------Author 1

\author[M. Ruzhansky]{Michael Ruzhansky}

\address{Department of Mathematics: Analysis, Logic and Discrete Mathematics
	\\
	Ghent University, Belgium
	\\
	and
	\\
	School of Mathematical Sciences
		\\ Queen Mary University of London 
			\\
		United Kingdom
		}

\email{Michael.Ruzhansky@ugent.be}

\author[Juan Pablo Velasquez-Rodriguez]{Juan Pablo Velasquez-Rodriguez}

\address{Department of Mathematics\\
Universidad del Valle\\
Calle 13 No 100-00, Cali\\
Colombia}

\email{velasquez.juan@correounivalle.edu.co}

%\thanks{The first author has been supported by **** . The second author has been supported  by. The third author has been supported by}

\subjclass[2010]{Primary; 58J40; Secondary: 47A10. }

\keywords{Non-harmonic analysis, Spectral theory, Pseudo-differential operators, Compact operators, Gershgorin theory, Fourier Analysis}

\thanks{The 
	authors were supported by the EPSRC Grant 
	EP/R003025/1, by the Leverhulme Research Grant RPG-2017-151, and by the FWO Odysseus grant.}
	
\date{\today}
\begin{abstract}
In this paper, %following \cite{N-H.AnalysisRT1, N-H.AnalysisRT2, DELGADONHAnalysis}, 
we use Operator Ideals Theory and  Gershgorin Theory to obtain explicit information concerning the spectrum of pseudo-differential operators, on a smooth manifold $\Omega$ with boundary $\partial \Omega$, in the context of the non-harmonic analysis of boundary value problems, introduced in \cite{N-H.AnalysisRT1} in terms of a model operator $\MO $. Under certain assumptions about the eigenfunctions of the model operator, for symbols in the Hörmander class $S^0_{1,0} (\Ombar \times \I )$,  we provide a ``non-harmonic version" of Gohberg's Lemma, and  a sufficient and necessary condition to ensure that the corresponding pseudo-differential operator is a compact operator in $L^2(\Omega)$. Also, for  pseudo-differential operators with symbols satisfying some integrability condition, one defines its associated matrix in terms of the biorthogonal system associated to $\MO$, and this matrix is used to give necessary and sufficient conditions for the $L^2(\Omega)$-boundedness, and to locate the spectrum of some operators. After that, we extend to the context of the non-harmonic analysis of boundary value problems the well known theorems about the exact domain of elliptic operators, and discuss some applications of the obtained results to evolution equations. Specifically we provide sufficient conditions to ensure the smoothness and stability of solutions to a generalised version of the heat equation. 
\end{abstract}

%%% ----------------------------------------------------------------------
\maketitle
%%% ----------------------------------------------------------------------
\tableofcontents
\section{{Introduction}}

Boundary value problems for pseudo-differential operators on manifolds or in domains $\Omega \subseteq \R^d$ with boundary $\partial \Omega$ have been studied in \cite{Abels, bvp2, bvp3, bvp4, bvp5, bvp6, bvp7, monvel1971, bvp8, bvp9, bvp10, bvp11, Mani1, Mani2, Mani3, Mani4} and references therein, among others. Most of the works in this topic exhibit a local approach, studying operators on the manifold through the use of charts. However, in some cases the analysis of pseudo-differential operators on a manifold can be simplified substituting the local approach by a global description \cite{DELGADONHAnalysis, Levyglobalsym, pseudosgeneratedbvp, globals1}  similar to the case of compact Lie groups \cite{clg1, clg2, ruzhansky1}.

The simplest example where a global analysis can be carried out is the $d$-dimensional torus $\Td:= \R^d /  2 \pi \Z^d$, where we have the concept of periodic pseudo-differential operators \cite{ruzhansky1, ruzhanskytor2008, Ruzhansky2007} developed with the aid of classical Fourier series techniques. The Fourier series on the unit circle $\T$, or more generally on any torus, can be viewed as an unitary transform in the Hilbert space $L^2(-\pi, \pi)$,  generated by the operator of differentiation $-i \frac{d}{dx}$ with periodic boundary conditions, because the system of exponential functions $\{e^{i x \cdot k}\}_{k \in \Z}$ is a system of its eigenfunctions. As it is exposed in \cite{N-H.AnalysisRT1, N-H.AnalysisRT2, DELGADONHAnalysis}
 this idea can be extended to a more general setting, without assuming that the problem has symmetries, using a differential operator $\MO$ of order $m$ with smooth coefficients, instead of the differential operator $-i \frac{d}{dx}$. In those papers the authors assume that $\MO$ is equipped with some boundary conditions, leading to a discrete spectrum, with its family of eigenfunctions yielding a Riesz basis in $L^2 (\Omega)$, which is a sequence $(x_n)_{n \in \N} \subseteq L^2 (\Omega)$ such that $\overline{Span\{x_n\}} = L^2 (\Omega)$ and $$ c \Big( \sum_n |a_n|^2 \Big) \leq ||\sum_n a_n x_n ||_{L^2 (\Omega)}^2 \leq C \Big( \sum_n |a_n|^2 \Big),  $$ for some constants $ 0<c \leq C < \infty$. This basis allows one to mimic the harmonic analysis constructions, and to carry out a global analysis similar to the toroidal case. The term ``non harmonic analysis" comes from the work of Paley and Wiener \cite{paley} who studied exponential systems $\{e^{2 \pi i x \cdot \lambda}\}_{\lambda \in \Lambda}$ on $L^2(0,1)$ for a discrete set $\Lambda$. Paley and Wiener use the term non-harmonic Fourier series to emphasize the distinction with the usual
(harmonic) Fourier series when $\Lambda = \Z$, and similarly in \cite{N-H.AnalysisRT1} the authors introduce the ``non-harmonic analysis of boundary value problems" as a (non-harmonic) Fourier analysis adapted to a boundary value problem.   
 
 The aim of this paper is to extend several results concerning the spectrum of pseudo-differential operators in the unit circle, to the context of the non-harmonic analysis of boundary value problems. Specifically, we will provide ``non-harmonic versions" of Gohberg's Lemma, compact operators characterisation, and spectrum localisation through Gershgorin Theory. For this, similar to \cite{N-H.AnalysisRT1, N-H.AnalysisRT2, DELGADONHAnalysis}, we have to make some assumptions about the model operator $\MO$. Throughout this paper we will be always working in the following setting:
 
 \hfill
 
\begin{asu}[\textbf{A}]
Let $\Omega$ be a smooth $d$-dimensional manifold with boundary $\partial \Omega$ such that $\Ombar$ is a compact (not necessarily smooth in the boundary) manifold. By $\MO_\Omega$ we denote a differential operator $\MO$
of order $m$ with smooth bounded coefficients in $\Omega$, equipped with some fixed linear boundary conditions. We assume that the boundary conditions called (BC) lead to a discrete spectrum, with a family of eigenfunctions yielding a Riesz basis in $L^2(\Omega)$. The discrete sets of eigenvalues and eigenfunctions will be indexed by a countable set $\I$ that, without loss of generality, will be always a subset of $\Z^l$ for some $l \in \N$ . We consider the spectrum $$Spec(\MO)=\{\lambda_\xi \in \C \esp : \esp  \xi  \in \I \},$$ of $\MO$ with
corresponding eigenfunctions in $L^2 (\Omega)$ denoted by $u_\xi$, i.e. $$\MO u_\xi = \lambda_\xi u_\xi \esp \text{in} \esp \Omega , \esp \text{for all} \esp \xi \in \I,$$ and the eigenfunctions $u_\xi$ satisfy the boundary conditions (BC). The conjugate spectral problem is $$\MO^* v_\xi = \overline{\lambda}_\xi v_\xi \esp \text{in} \esp \Omega , \esp \text{for all} \esp \xi \in \I,$$ which we equip with the conjugate boundary conditions $(BC)^*$. We assume that the functions $u_\xi$,$v_\xi$ are normalised $$\int_\Omega |u_\xi (x)|^2 dx = \int_\Omega |v_\xi (x)|^2 dx = 1 , \esp \esp \text{for all} \esp \xi \in \I, $$ and that $$ \sup_{x \in \Omega} |u_\xi (x)|\leq C_b \langle \xi \rangle^{\mu_0},  $$ for some constants $C_b >0$, $\mu_0 >0$ and every $\xi \in \I$. Here we have used the notation $$\langle \xi \rangle := (1 + |\lambda_\xi|^2)^{\frac{1}{2m}}, $$ where $m$ is the order of the differential operator $\MO$. Recall that the systems $\{u_\xi \}_{\xi \in \I}$ and $\{v_\xi \}_{\xi \in \I}$ are biorthogonal, i.e. $$( u_\xi , v_\eta )_{L^2(\Omega)} = \delta_{\xi\eta},$$ where $$( f , g )_{L^2(\Omega)} := \int_{\Omega} f(x) \overline{g(x)} dx,$$ is the usual inner product of the Hilbert space $L^2(\Omega)$ and $dx$ a measure on $\Omega$. If $\Omega$ has finite measure then we assume that the measure is normalised.  
\end{asu}

 By associating a discrete Fourier analysis to the system $\{ u_\xi \}_{\xi \in \I} $, the authors in \cite{N-H.AnalysisRT1} introduced a full symbol for a given operator acting on suitable functions over $\Omega \subset \R^d$, and this development has already been extended to smooth manifolds with boundary in \cite{DELGADONHAnalysis}. We will recall the basic elements of such symbolic analysis in Section 3.

\hfill

This paper is organised as follows:

\begin{itemize}
    \item Section 2: we give examples of operators $\MO$, and different boundary conditions yielding different types of biorthogonal systems. 
    \item Section 3: we recall the basic elements of the discrete Fourier analysis, quantisation and full symbols associated to the system of eigenfunctions of a model operator $\MO$. 
    \item Section 4: assuming Gohberg's Lemma in $L^2(\Omega)$, we provide a sufficient and necessary condition to ensure the compactness of a global pseudo-differential operator with symbol in the Hörmander class $S^0_{1,0} (\Ombar \times \I )$.
    \item Section 5: we will show that the spectrum of some pseudo-differential operators can be localised with the aid of Gershgorin Theory. Also, we will discuss an aplication of this spectrum localisation to evolution equations. 
    \item Section 6: we provide a proof of Gohberg's Lemma in $L^2(\Omega)$.
\end{itemize}

\hfill

 \section{{Examples of operators $\MO$ and boundary conditions}}
 
In this section we give several examples of the operator $\MO$ satisfying Assumption (A) and of boundary conditions $(BC)$. We want to remark that the property of having real-valued eigenfunctions will be of importance for the analysis in Section 5. For more examples see \cite{N-H.AnalysisRT1}. 

\hfill

\begin{exa}
Let $\Omega= (0,2\pi)^d$. Define $\MO_d$ in $\Omega$ as the differential operator $$\MO_d := \sum_{j=1}^d \frac{\partial^2}{\partial x_j^2},$$
together with periodic boundary conditions. This operator is self-adjoint with the domain $W_2^2 (\Omega)$ and its system of eigenfunctions is $$\{e^{i x \cdot \xi} \esp : \esp \xi \in \Z^d\},$$ which form, with a proper choise of measure, an orthonormal basis of $L^2 (\Omega)$. Eigenvalues of $\MO_d$ are $$\{ -|\xi|^2 \esp : \esp \xi \in \Z^d \}.$$
Recall that we can identify the functions in $(0,2\pi)^d$ that satisfy periodic boundary conditions with functions on the $d$-dimensional torus $\Td$. Clearly, eigenvalues and eigenfunctions of $\MO_d$ from both perspectives coincide, and also  satisfy Assumption (A). If we restrict our attention to real-valued functions, the periodic boundary value problem leads to the orthonormal basis of real-valued eigenfunctions $$\{\sqrt{2} \sin{(x \cdot \xi)} , \sqrt{2} \cos{(x \cdot \xi)}\}_{\xi \in \Z^d}.$$ 
\end{exa}

\begin{exa}
Similar to the previous example, let $\Omega= (0,2\pi)^d$ and let $h \in \R^d$ be such that $h_j >0$ for $1 \leq j \leq d$. Define $\MO_{h,d}$ in $\Omega$ as the differential operator $$\MO_{h,d} := \sum_{j=1}^d \frac{\partial^2}{\partial x_j^2},$$
together with the boundary conditions $(BC)$: \begin{align}
     f(x) |_{x_j=0} = h_jf(x)|_{x_j = 2 \pi} , \esp  \frac{\partial f}{\partial x_j} (x)|_{x_j = 0} = h_j \frac{\partial f}{\partial x_j} (x)|_{x_j = 2 \pi} ,  \esp \esp 1 \leq j \leq d, \esp
\end{align}and the domain $$Dom(\MO_{h,d}) = \{f \in L^2(\Omega) \esp : \esp \MO_{h,d} f \in L^2(\Omega) \esp \text{and} \esp f \esp \text{satisfies (1)}\},$$
Then, with $\I = \Z^d$, the system of eigenfunctions of the operator $\MO_{h,d}$ is $$\{u_\xi (x) = h^{x/2\pi} e^{i x \cdot \xi} \esp : \esp \xi \in \I\},$$ and the conjugate system is $$\{v_\xi (x) = h^{-{x/2\pi}} e^{i x \cdot \xi} \esp : \esp \xi \in \I\},$$ where $$h^{x/2\pi} := \prod_{j=1}^d h_j^{x_j/2\pi},$$ 

See \cite{N-H.AnalysisRT1} and the references therein for a detailed treatment. 
\end{exa}

\begin{exa}
The real-valued analogue of the above example is the operator $$\MO_{h,d} := \sum_{j=1}^d \frac{\partial^2}{\partial x_j^2} -  \frac{\ln{h_j}}{ \pi} \frac{\partial}{\partial x_j},$$ with the same boundary conditions as in the previous example. This operator leads to the basis of eigenfunctions $$\{ \sqrt{2} h^{x/2\pi} \cos(x\cdot \xi), \sqrt{2} h^{x/2\pi} \sin(x \cdot \xi): \esp \xi \in \Z^d\},$$ with the corresponding eigenvalues $$ - |\xi|^2 - \frac{1}{ 4 \pi^2} \sum_{j=1}^d (\ln{h_j})^2,$$ and with the corresponding biorthogonal system $$\{ \sqrt{2} h^{-x/2\pi} \cos(x\cdot \xi), \sqrt{2} h^{-x/2\pi} \sin(x \cdot \xi): \esp \xi \in \Z^d\}.$$
\end{exa}

\hfill

\begin{exa} Let $\Omega = (0,2 \pi a) \times (0,2 \pi b)$ with $a>0$ and $b>0$. Define $\MO$ as the two dimensional Laplace operator with domain $W_2^2 (\Ombar)$ and Neumann boundary conditions. As it is well known this operator is self-adjoint and its system of eigenfunctions is $$\{u_{nm}(x,y)= \sqrt{2} \cos{\frac{n x}{a}} \cdot \cos{\frac{m y}{b}} : \esp m,n \in \N\},$$ which is an orthonormal basis of $L^2 (\Omega)$. Thus $\MO$ satisfies Assumption (A).
\end{exa}

\hfill

\begin{exa}
Let $\Omega=(0,1)$. Define $\MO$ as the differential operator $$\MO:= - \frac{d^2}{dx^2},$$ on the domain  $$Dom(\MO)= \{f \in W^2_2 [0,1] : \esp f(0)=0 \esp , \frac{df}{dx}(0) = \frac{df}{dx}(1)\}.$$ This operator was studied in detail in \cite{Ionkin, Kanguzhin-Tokmagambetov}. The system of eigenfunctions of $\MO$ is $$u_0 (x) = x , \esp u_{2k-1}(x) = \sin (2\pi kx), \esp u_{2k} (x) = x \cos (2 \pi kx), \esp \esp k \in \N,$$ and the adjoint functions are $$v_0 (x) = 2 , \esp v_{2k-1}(x) =4(1-x) \sin (2\pi kx), \esp v_{2k} (x) = 4 \cos (2 \pi kx) \esp \esp k \in \N.$$ Since this system is a biorthogonal basis of $L^2 (0,1)$, the operator $\MO$ satisfies Assumption (A). 
\end{exa} 

\hfill

\begin{exa}
Let $\Omega = (0,1)$. Consider the operator $\MO = -i \frac{d}{dx}$ with the domain $$Dom(\MO)= \Big\{ f \in W^1_2 [0,1]: af(0) + bf(1) + \int_0^1 f(x) q(x) dx = 0 \Big\},$$ where $a,b \neq 0$ and $q \in C^1[0,1]$. We assume that $$a+b+ \int_0^1 q(x)dx = 1,$$ so that the inverse $\MO^{-1}$ exists and is bounded. Following \cite{Kanguzhin} we have that the system of extended eigenfunctions of $\MO$ is $$\Big\{ u_{jk} (x) = \frac{(ix)^k}{k!} e^{i \lambda_j x} : \esp 0 \leq k \leq m_j -1, \esp j \in \Z \Big\},$$ where $m_j$ denotes the multiplicity of the eigenvalue $\lambda_j= -i \ln(-a/b) + 2j\pi + \alpha_j,$ and for any $\varepsilon > 0$ we have $$\sum_{j \in \Z} |\alpha_j|^{1 + \varepsilon} < \infty. $$ Its biorthogonal system is given by $$v_{jk}(x) = \lim_{\lambda \to \lambda_j} \frac{1}{k!} \frac{d^k}{d \lambda^k} \Big( \frac{(\lambda - \lambda_j)^{m_j}}{\Delta(\lambda)} (ib e^{i \lambda(1-x)} + i \int_x^1 e^{i \lambda (t-x)}q(t) dt) \Big),$$ where $$\Delta(\lambda) = a + b e^{i \lambda} + \int_0^1 e^{i \lambda x} q(x) dx.$$ 
\end{exa}

\hfill

\begin{exa}
Let $\Omega = (-\pi,\pi) \times (0, \pi)$. Define $\MO$ as the operator $$\MO f = \frac{1}{\sin{(x_2)}}\frac{\partial}{\partial x_2}(\sin(x_2) \frac{\partial f}{\partial x_2}) + \frac{1}{\sin^2 (x_2)}\frac{\partial^2 f}{\partial x_1^2},$$ together with the boundary conditions $(BC)$:
\begin{enumerate}
    \item[(i)] $f(x_1,0) = c_1,$
    \item[(ii)] $f(x_1, \pi) = c_2,$
    \item[(iii)] $f(-\pi , x_2) = f(\pi , x_2)$ for all $x_2 \in (0,\pi)$.
\end{enumerate}
 
Similar to the periodic case, a function that satisfies $(BC)$ can be identified with a function on the sphere $\mathbb{S}^2$. Thus $\MO$ is self-adjoint in the weighted Lebesgue space$$L^2(\Omega , dx') = \{f: \Omega \to \C : \esp \int_{\Omega} |f(x_1,x_2)|^2 \sin(x_2) dx_2 dx_1 < \infty \},$$ where $dx' = \sin(x_2) dx_2 dx_1$. Its corresponding orthonormal basis of eigenfunctions is the collection of spherical harmonics $$u_{ml}(x_1 , x_2) = \sqrt{\frac{(2l+1)(l-m)!}{4 \pi (l+m )!}} P_l^m (\cos(x_2)) e^{i x_1 m}, \esp \esp l \in \N , m \in \Z,$$ with eigenvalues $l(l+1)$, where $P_l^m$ is the corresponding associated polynomial of Legendre. 
If we restrict our attention to real-valued functions, the boundary value problem leads to the orthonormal basis of real eigenfunctions \[   
u_{ml}(x_1,x_2) = 
     \begin{cases}
        (-1)^m (2 \pi) \sqrt{\frac{(2l+1)(l-|m|)!}{4 \pi (l+|m| )!}} P_l^{|m|} (\cos(x_2)) \sin{|m|x_1}, & \text{if} \esp m<0,\\ (2 \pi^2) \sqrt{\frac{(2l+1)}{4 \pi}} P_l^m (\cos{(x_2)}), &\text{if} \esp m=0, \\ (-1)^m (2 \pi) \sqrt{\frac{(2l+1)(l-m)!}{4 \pi (l+m )!}} P_l^m (\cos(x_2)) \cos{m x_1}, & \text{if} \esp m>0 .\\ 
     \end{cases}
\]

\end{exa}

\begin{exa}
Let $\Omega = (0,2 \pi) \times (0, \pi)$. Combining Examples 2.3 and 2.7 we can consider the operator $$\MO_h f := \frac{1}{\sin{(x_2)}}\frac{\partial}{\partial x_2}(\sin(x_2) \frac{\partial f}{\partial x_2}) + \frac{1}{\sin^2 (x_2)}\Big( \frac{\partial^2 f}{\partial x_1^2} - \frac{\ln h}{ \pi} \frac{\partial f}{\partial x_1} +\frac{(\ln h )^2}{4 \pi } f  \Big),$$ together with the boundary conditions $(BC)$
\begin{enumerate}
    \item[(i)] $f(x_1,0) = c_1,$
    \item[(ii)] $f(x_1, \pi) = c_2,$
    \item[(iii)] $f(0 , x_2) = h f(2 \pi , x_2)$ for all $x_2 \in (0,\pi)$.
\end{enumerate}
The operator $\MO_h$ has a discrete spectrum and its eigenvalues are $$l(l+1)$$ with corresponding eigenfunctions \[   
u_{ml}(x_1,x_2) = 
     \begin{cases}
        (-1)^m  (2 \pi) \sqrt{\frac{(2l+1)(l-|m|)!}{4 \pi (l+|m| )!}} P_l^{|m|} (\cos(x_2)) h^{\frac{x_1}{2\pi}} \sin{|m|x_1}, & \text{if} \esp m < 0, \\ (2 \pi^2) \sqrt{\frac{(2l+1)}{2}} P_l^m (\cos{(x_2)}) h^{\frac{x_1}{2 \pi}}, & \text{if} \esp m =0,\\ (-1)^m (2 \pi) \sqrt{\frac{(2l+1)(l-m)!}{4 \pi (l+m )!}} P_l^m (\cos(x_2)) h^{\frac{x_1}{2 \pi}}\cos{m x_1}, & \text{if} \esp m>0,\\ 
     \end{cases}
\] and the corresponding biorthonormal system \[   
v_{ml}(x_1,x_2) = 
     \begin{cases}
        (-1)^m  (2\pi) \sqrt{\frac{(2l+1)(l-|m|)!}{4 \pi (l+|m| )!}} P_l^{|m|} (\cos(x_2)) h^{\frac{-x_1}{2\pi}} \sin{|m|x_1}, & \text{if} \esp m < 0,\\ (2 \pi^2)\sqrt{\frac{(2l+1)}{2}} P_l^m (\cos{(x_2)}) h^{\frac{-x_1}{2 \pi}}, & \text{if} \esp m =0,  \\ (-1)^m (2 \pi)  \sqrt{\frac{(2l+1)(l-m)!}{4 \pi (l+m )!}} P_l^m (\cos(x_2)) h^{\frac{-x_1}{2 \pi}}\cos{m x_1}, & \text{if} \esp m>0.\\ 
     \end{cases}
\]
\end{exa}

\begin{exa}
Let $\Omega =(-\pi, \pi) \times (-\pi/2 , \pi/2)$. Consider the linear operator $$\MO:= \frac{\partial^2}{\partial x^2} + \frac{\partial^2}{\partial y^2},$$ together with the boundary conditions \begin{enumerate}
    \item[(i)] $f(- \pi , y ) = f(\pi , -y)$ and $\frac{\partial f}{\partial x}(- \pi , y ) = \frac{\partial f}{\partial x}( \pi , -y )=0 $ for all $y \in (-\pi , \pi)$,
    \item[(ii)] $f(x,- \pi/2) = f(x,\pi/2) = 0$ for all $x \in (- \pi , \pi).$
\end{enumerate}

Functions that satisfy the first item of the above boundary conditions can be identified with functions in the Möbius strip. The second item determines a Dirichlet boundary condition in the Möbius strip. With this boundary conditions the operator $\MO$ is self-adjoint and, using separation of variables, one can see that it has an orthonormal basis of eigenfunctions given by \[   
u_{mn}(x,y) = \frac{1}{\sqrt{2}} \sin( (\frac{2m+1}{2})x)\sin(2ny).
\]

\end{exa}

\section{{Preliminaries}}
In this section we recall the basics on the discrete Fourier analysis associated to the system of eigenfunctions of a model operator $\MO$ introduced in \cite{N-H.AnalysisRT1, N-H.AnalysisRT2, DELGADONHAnalysis}. In what follows,  $\mathcal{L} (E,F)$ will denote the collection of all continous linear operators from $E$ to $F$, the Fréchet spaces. For $E=F$  we write $\mathcal{L}(E)$ instead of $\mathcal{L} (E,E)$. 

\subsection{Test functions for $\MO$ and Schwartz kernel} In this subsection we recall some spaces of distributions generated by $\MO$ and by its adjoint $\MO^*$. We also recall the version of the Schwartz kernel theorem corresponding to the present framework. 

\begin{defi}\normalfont
\normalfont{ The space $C^\infty_\MO (\Ombar):= Dom (\MO^\infty)$ is called} the space of test functions for $\MO$.\normalfont{ Here, as in \cite{N-H.AnalysisRT1}, it is defined by $$Dom(\MO^\infty) := \bigcap_{k = 1 }^{\infty} Dom(\MO^k),$$ where $Dom(\MO^k)$ is the domain of the operator $\MO^k$, in turn defined as $$Dom(\MO^k) := \{ f \in L^2 (\Omega) : \esp \MO^j f \in Dom(\MO) ,  \esp j= 0, 1, ..., k-1\}.$$ The Fréchet topology of  $C^\infty_\mathfrak{L} (\Ombar)$ is given by the family of norms $$\norm{\varphi}_{C_\MO^k} := \max_{0 \leq j \leq k} \norm{\MO^j \varphi }_{L^2 (\Omega)}, \esp \esp k \in \N_0 , \esp \varphi \in  C^\infty_\MO (\Ombar).$$  

Analogously to the $\MO$-case, the space  $C^\infty_{\MO^*} (\Ombar)$ corresponding to the adjoint operator $\MO^*$ is defined by $$C^\infty_{\MO^*} (\Ombar) := Dom((\MO^*)^\infty) = \bigcap_{k=1}^{\infty} Dom((\MO^*)^k),$$ where $Dom((\MO^*)^k)$ is the domain of the operator $(\MO^*)^k$ $$Dom((\MO^*)^k):= \{ f \in L^2 (\Omega) : \esp (\MO^*)^j f \in Dom(\MO^*) ,  \esp j= 0, 1, ..., k-1\},$$ which satisfy the adjoint boundary conditions corresponding to the operator $\MO_\Omega^*$. The Fréchet topology of  $C^\infty_{\MO^*} (\Ombar)$ is given by the family of norms $$\norm{\psi}_{C_{\MO^*}^k} := \max_{0 \leq j \leq k} \norm{(\MO^*)^j \psi }_{L^2 (\Omega)}, \esp \esp k \in \N_0 , \esp \varphi \in  C^\infty_{\MO^*} (\Ombar).$$} 
\end{defi}

\begin{rem}

Since we have $u_\xi \in C^\infty_\MO (\Ombar)$ and $v_\xi \in C^\infty_{\MO^*}(\Ombar)$ for all $\xi \in \I$, we observe that Assumption (A) implies that the spaces $C^\infty_\MO (\Ombar)$ and $C^\infty_{\MO^*}(\Ombar)$ are dense in $L^2 (\Omega)$. %We also note that $C^\infty_\MO (\Ombar), C^\infty_{\MO^*} (\Ombar) \subseteq C^\infty (\Omega)$, because $\MO$ is a differential operator of order $m$ so $$Dom(\MO) \subseteq C^m (\Omega) .$$ 
\end{rem}

\begin{defi}\normalfont\normalfont
\normalfont The space $$\mathcal{D}_\MO' (\Omega) := \mathcal{L}(C_{\MO*}^\infty (\Ombar) , \C),$$ of linear continuous functionals on $C_{\MO*}^\infty (\Ombar)$ is called the space of $\MO$-distributions. Analogously the space $$\mathcal{D}_{\MO^*}'(\Omega) := \mathcal{L}(C_{\MO}^\infty (\Ombar), \C),$$ of linear continuous functionals on $C^\infty_\MO (\Omega)$ is called the space of $\MO^*$-distributions.
\end{defi}

\begin{rem}
For any $\psi \in C^\infty_\MO (\Ombar) $, $$C^\infty_{\MO^*} (\Ombar) \ni \varphi \mapsto \int_\Omega \psi(x) \varphi (x) dx $$ is an $\MO$-distribution, which gives an embedding $C^\infty_\MO (\Ombar) \xhookrightarrow{} \mathcal{D}_\MO' (\Omega)$.
\end{rem}
Now we recall the Schwartz kernel theorem. For this we need the following:

\begin{asu}[\textbf{B}]
Assume that the number $s_0 >0$ is such that we have $$\sum_{\xi \in \I} \langle \xi \rangle^{-s_0}< \infty.$$ 
\end{asu}

We will use the notation $$C^\infty_\MO (\Ombar \times \Ombar) := C^\infty_\MO (\Ombar) \overline{\otimes} C^\infty_\MO (\Ombar),$$ and $$C^\infty_{\MO^*} (\Ombar \times \Ombar) := C^\infty_{\MO^*} (\Ombar) \overline{\otimes} C^\infty_{\MO*} (\Ombar),$$ with the Fréchet topologies given by the family of tensor norms $$\norm{\varphi \otimes \psi}_{C^k_\MO(\Ombar \times \Ombar)} := \max_{0 \leq j + l \leq k } \norm{\MO^j \varphi}_{L^2(\Omega)} \norm{\MO^l \psi}_{L^2(\Omega)}, \esp \esp k \in \N_0, $$ and  $$\norm{\varphi \otimes \psi}_{C^k_{\MO^*}(\Ombar \times \Ombar)} := \max_{0 \leq j + l \leq k } \norm{(\MO^*)^j \varphi}_{L^2(\Omega)} \norm{(\MO^*)^l \psi}_{L^2(\Omega)}, \esp \esp k \in \N_0.$$ For the corresponding dual spaces we write \begin{align*}
    \mathcal{D}_\MO' (\Omega \times \Omega) := (C^\infty_{\MO^*} (\Ombar \times \Ombar))', \\ \mathcal{D}_{\MO^*}' (\Omega \times \Omega) := (C^\infty_{\MO} (\Ombar \times \Ombar))'.  
\end{align*} 

\begin{teo}[Schwartz kernel]
For any linear continuous operator $$A: C^\infty_\MO (\Ombar) \to \mathcal{D}_\MO' (\Omega ),$$ there exists a kernel $K_A \in \mathcal{D}_\MO' (\Omega \times \Omega)$ such that for all $f \in C^\infty_\MO (\Ombar)$, we can write, in the distribution sense $$A f (x) = \int_\Omega K_A(x,y) f(y) dy.$$ 
Also, for any linear continuous operator $$A: C^\infty_{\MO^*} (\Ombar) \to \mathcal{D}_{\MO^*}' (\Omega)$$ there exists a kernel $\Tilde{K}_A \in \mathcal{D}_{L^*} '(\Omega \times \Omega)$ such that for all $f \in C^\infty_{\MO^*} (\Ombar)$ we can write, in the distribution sense $$A f (x) = \int_\Omega \Tilde{K}_A(x,y) f(y) dy.$$ 
\end{teo}

For further discussion see \cite{N-H.AnalysisRT1, N-H.AnalysisRT2}.

\subsection{$\MO$-Fourier transform} In this subsection we recall the definition of the $\MO$-Fourier transform. 

Let $\mathcal{ S}(\I)$ denote the space of rapidly decaying functions $\varphi : \I \to \C$. That is, $\varphi \in \mathcal{S} (\I)$ if for any $M < \infty$ there exists a constant $C_{\varphi , M}$ such that $$|\varphi (\xi)| \leq C_{\varphi , M} \langle \xi \rangle^{-M},$$ holds for all $\xi \in \I$. The topology in $\mathcal{S}(\I)$ is given by the seminorms $p_k$ where $k \in \N_0$ and $$p_k (\varphi) := \sup_{\xi \in \I} \langle  \xi \rangle^k |\varphi(\xi)|.$$ Continuous linear functionals on $\mathcal{S}(\I)$ are of the form $$\varphi \mapsto \langle u , \varphi \rangle := \sum_{\xi \in \I} u(\xi) \varphi (\xi),$$ where functions $u : \I \to \C$ grow at most polynomially at infinity i.e. there exist constants $M < \infty$ and $C_{u,M}$ such that $$|u(\xi)| \leq C_{u,M} \langle \xi \rangle^{M},$$ holds for all $\xi \in \I$. Such distributions $u:\I \to \C$ form the space of distributions which we denote by $\mathcal{S}'(\I)$.

\begin{defi}\normalfont\normalfont
The $\MO$-Fourier transform $$(\mathcal{F}_\MO  f)(\xi) = (f \mapsto \widehat{f}): C^\infty_\MO (\Ombar) \to \mathcal{S}(\I),$$ is defined by $$\widehat{f}(\xi):= (\mathcal{F}_\MO f)(\xi) = \int_\Omega f(x) \overline{v_\xi (x)} dx.$$ Analogously, one defines the $\MO^*$-Fourier transform $$(\mathcal{F}_{\MO^*}  f)(\xi) = (f \mapsto \widehat{f}_*): C^\infty_{\MO^*} (\Ombar) \to \mathcal{S}(\I),$$ by $$\widehat{f}_* (\xi):= (\mathcal{F}_{\MO^*} f)(\xi) = \int_\Omega f(x) \overline{u_\xi (x)} dx.$$
\end{defi}

The next proposition can be found in \cite[Proposition 2.7]{N-H.AnalysisRT1}.

\begin{pro}
The $\MO$-Fourier transform $\mathcal{F}_\MO$ is a bijective homeomorphism from $C^\infty_\MO (\Ombar)$ to $ \mathcal{S}(\I)$. Its inverse $$\mathcal{F}^{-1}_\MO : \mathcal{S}(\I)  \to C^\infty_\MO (\Ombar),$$ is given by $$(\mathcal{F}_\MO^{-1} h )(x) = \sum_{\xi \in \I} h (\xi) u_\xi (x), \esp \esp h \in \mathcal{S}(\I),$$ so that the Fourier inversion formula becomes $$f(x) = \sum_{\xi \in \I} \widehat{f}(\xi) u_\xi (x) \esp \esp \text{for all} \esp f \in C^\infty_\MO (\Ombar).$$ Similarly, $\mathcal{F}_{\MO^*} : C^\infty_{\MO^*} (\Ombar) \to \mathcal{S}(\I)$ is a bijective homeomorphism and its inverse $$\mathcal{F}_{\MO^*}^{-1} : \mathcal{S}(\I) \to C^\infty_{\MO^*} (\Ombar),$$ is given by $$(\mathcal{F}_{\MO^*}^{-1} h)(x):= \sum_{\xi \in \I} h(\xi) v_\xi (x), $$ so that the conjugate Fourier inversion formula becomes $$f(x) = \sum_{\xi \in \I} \widehat{f}_* (\xi) v_\xi (x) \esp \esp \text{for all } \esp f \in C^\infty_{\MO^*} (\Ombar).$$
\end{pro}

We note that since the systems of $u_\xi$ and of $v_\xi$ are Riesz bases, we can also compare the $L^2$-norms of functions with sums of squares of Fourier coefficients. The following statement follows from the work of Bari \cite{Bar51}.

\begin{lema}
There exist constants $k_1,K_1,k_2,K_2 >0$ such that for every $f \in L^2 (\Omega)$ we have $$k_1^2 \norm{f}_{L^2 (\Omega)}^2 \leq \sum_{\xi \in \I} |\widehat{f} (\xi)|^2 \leq K_1^2 \norm{f}_{L^2(\Omega)}^2,$$ and $$k_2^2 \norm{f}_{L^2(\Omega)}^2 \leq \sum_{\xi \in \I} |\widehat{f}_* (\xi)|^2 \leq K_2^2 \norm{f}_{L^2(\Omega)}^2.$$
\end{lema}

However we note that the Plancherel identity can be also achieved in suitably defined $\ell^2$-spaces of Fourier coefficients, see Proposition 3.10.

\subsection{Plancherel formula and Sobolev spaces} In this subsection we recall the Plancherel identity obtained by defining suitable sequence spaces $\ell^2 (\MO)$ and $\ell^2(\MO^*)$ adapted to the present framework. Also, we recall the definition of Sobolev spaces associated to the model operator $\MO$. 

\begin{defi}\normalfont
We will denote by $\ell^2 (\MO)$ the linear space of complex valued functions $a$ on $\I$ such that $\mathcal{F}_\MO^{-1} a \in L^2 (\Omega)$, i.e. if there exists $f \in L^2 (\Omega)$ such that $\mathcal{F}_\MO f = a$. Then the space of sequences $\ell^2 (\MO)$ is a Hilbert space with the inner product $$( a , b )_{\ell^2 (\MO)} := \sum_{\xi \in \I} a(\xi) \overline{(\mathcal{F}_{\MO^*} \circ \mathcal{F}_\MO^{-1} b) (\xi)},$$ for arbitrary $a,b \in \ell^2(\MO).$  Analogously, the Hilbert space $\ell^2 (\MO^*)$ is the space of functions $a$ on $\I$ such that $\mathcal{F}_{\MO^*}^{-1} a \in L^2 (\Omega)$, with the inner product $$( a , b )_{\ell^2 (\MO^*)} := \sum_{\xi \in \I} a(\xi) \overline{(\mathcal{F}_{\MO} \circ \mathcal{F}_{\MO^*}^{-1} b) (\xi)}.$$ Also, we recall the definition of the $\ell^p$-spaces (see \cite[Definition 7.1]{N-H.AnalysisRT1}) associated with the model operator $\MO$ defined by $$\ell^p(\MO) := \{a:\I \to \C: \esp \sum_{\xi \in \I} |a(\xi)|^p ||u_\xi||_{L^\infty (\Omega)}^{2-p} < \infty \},$$ $$\ell^p(\MO^*) := \{a:\I \to \C: \esp \sum_{\xi \in \I} |a(\xi)|^p ||v_\xi||_{L^\infty (\Omega)}^{2-p} < \infty \},$$ for $1\leq p \leq 2$, and $$\ell^p(\MO) := \{a:\I \to \C: \esp \sum_{\xi \in \I} |a(\xi)|^p ||v_\xi||_{L^\infty (\Omega)}^{2-p} < \infty \},$$ $$\ell^p(\MO^*) := \{a:\I \to \C: \esp \sum_{\xi \in \I} |a(\xi)|^p ||u_\xi||_{L^\infty (\Omega)}^{2-p} < \infty \},$$ for $2\leq p <\infty $. Also, we recall the definition of the usual $\ell^p$-spaces $$\ell^p(\I) := \{a:\I \to \C: \esp \sum_{\xi \in \I} |a(\xi)|^p < \infty \},$$ for $1\leq p < \infty$.
\end{defi}

The reason for the definition in the above form becomes clear in view of the following Plancherel identity. See \cite[Proposition 6.1]{N-H.AnalysisRT1}.

\begin{pro}[Plancherel's identity] If $f,g \in L^2 (\Omega)$ then $\widehat{f}, \widehat{g} \in \ell^2 (\MO)$, $\widehat{f}_*, \widehat{g}_* \in \ell^2 (\MO^*)$ and the inner products take the form $$( \widehat{f} , \widehat{g} )_{\ell^2 (\MO)} = \sum_{\xi \in \I} \widehat{f} (\xi) \overline{\widehat{g}_* (\xi)},$$ and $$( \widehat{f}_* , \widehat{g}_* )_{\ell^2 (\MO^*)} = \sum_{\xi \in \I} \widehat{f}_* (\xi) \overline{\widehat{g} (\xi)}.$$ In particular we have $$\overline{( \widehat{f} , \widehat{g} )}_{\ell^2 (\MO)} = ( \widehat{g}_* , \widehat{f}_* )_{\ell^2 (\MO^*)}.$$ The Parseval identity takes the form$$( f,g )_{L^2(\Omega)} = ( \widehat{f} , \widehat{g} )_{\ell^2 (\MO)} = \sum_{\xi \in \I} \widehat{f} (\xi) \overline{\widehat{g} (\xi)}.$$ Furthermore, for any $f \in L^2 (\Omega)$, we have $\widehat{f} \in \ell^2 (\MO)$, $\widehat{f}_* \in \ell^2 (\MO^*)$, and $$\norm{f}_{L^2(\Omega)} = ||\widehat{f}||_{\ell^2 (\MO)} = ||\widehat{f}_*||_{\ell^2 (\MO^*)}.$$
\end{pro}

As a consequence of the properties of the $\MO$-Fourier transform collected so far, the definition of Sobolev space correspondent to the present setting naturally arises \cite{N-H.AnalysisRT1}.

\begin{defi}\normalfont[Sobolev spaces $\mathcal{H}^s_\MO (\Omega)$]
For $f \in \mathcal{D}_\MO' (\Omega) \cap \mathcal{D}_{\MO^*}' (\Omega)$ and $s \in \R$, we say that $f \in \mathcal{H}^s_\MO (\Omega)$ if and only if $\langle \xi \rangle^s \widehat{f} (\xi) \in \ell^2 (\MO)$. We define the norm on $\mathcal{H}^s_\MO (\Omega)$ by $$||f||_{\mathcal{H}^s_\MO (\Omega)} := \Big( \sum_{\xi \in \I} \langle \xi \rangle^{2s} \widehat{f} (\xi) \overline{\widehat{f}_* (\xi)} \Big)^{1/2}.$$ The Sobolev space $\mathcal{H}^s_\MO (\Omega)$ is then the space of $\MO$-distributions $f$ for which we have $||f||_{\mathcal{H}^s_\MO} < \infty$. Similarly, we can define the space $\mathcal{H}^s_{\MO^*} (\Omega)$ by the condition  $$||f||_{\mathcal{H}^s_{\MO^*} (\Omega)} := \Big( \sum_{\xi \in \I} \langle \xi \rangle^{2s} \widehat{f}_* (\xi) \overline{\widehat{f} (\xi)} \Big)^{1/2}< \infty.$$ We note that $\mathcal{H}_{\MO}^s= \mathcal{H}_{\MO^*}^s.$
\end{defi}
\subsection{$\MO$-admissible operators and $\MO$-quantisation}

In this subsection we describe the $\MO$-quantisation of the $\MO$–admissible operators induced
by the operator $\MO$.

\begin{defi}\normalfont
We say that the linear continuous operator $$A: C^\infty_\MO (\Ombar) \to \mathcal{D}_\MO'(\Omega),$$ 
belongs to the class of $\MO$–admissible operators if $$\sum_{\eta \in \I} u_\eta^{-1} (x) u_\eta (z) \int_{\Omega} K_A (x,y) u_\eta (y) dy ,$$ is in $\mathcal{D}_\MO'(\Omega \times \Omega)$. For example, this is the case when the functions $u_\xi$ do not have any zeros in $\Omega$.
\end{defi}

\begin{rem}
Note that the expression
$$u_\eta^{-1} (x) \int_\Omega K_A (x,y) u_\eta (y) dy,$$
exists for any operator $A$ from the class of $\MO$–admissible operators. Moreover, it is in $\mathcal{D}_\MO'(\Omega) \otimes \mathcal{S}'(\I)$.
\end{rem}

\begin{defi}\normalfont[$\MO$-Symbols of operators]
The $\MO$-symbol of a linear continuous $\MO$–admissible operator $$A: C^\infty_\MO (\Ombar) \to \mathcal{D}_\MO'(\Omega),$$ is defined by $$\sigma_A (x,\xi) := u_\xi^{-1} (x) \int_\Omega K_A (x,y) u_\xi (y) dy.$$
\end{defi}

\begin{teo}
Let $$A: C^\infty_\MO (\Ombar) \to \mathcal{D}_\MO'(\Omega),$$ be a linear continuous $\MO$–admissible operator with $\MO$-symbol $\sigma_A \in \mathcal{D}_\MO'(\Omega) \otimes \mathcal{S}'(\I)$. Then
the $\MO$–quantisation $$Af (x) = \sum_{\xi \in \I} \sigma_A (x , \xi) \widehat{f} (\xi) u_\xi (x),$$ is true for every $f \in C^\infty_\MO (\Ombar)$. The $\MO$-symbol of $A$ can be written as $$\sigma_A (x,\xi) = u_\xi^{-1} (x) A u_\xi (x).$$
\end{teo}

In virtue of the above theorem, from now on we will be interested mainly in operators $A: C^\infty_\MO (\Ombar) \to \mathcal{D}_\MO'(\Omega)$ from the
class of $\MO$–admissible operators. However, in some cases we will consider a larger class. This is explained in the following remark.

\begin{rem}
Let $$A: Span\{u_\xi\} \subseteq  Dom(A)  \subseteq L^2(\Omega) \to L^2 (\Omega),$$ be a linear operator. If there exist a measurable function $\sigma_A : \Ombar \times \I \to \C$ such that \begin{align*}
    \sigma_A (x,\xi) u_\xi (x) = A u_\xi(x ), \tag{2}
\end{align*} then we note that the $\MO$-quantisation $$Af (x) = \sum_{\xi \in \I} \sigma_A (x , \xi) \widehat{f} (\xi) u_\xi (x), $$ is true for every $f \in Span\{u_\xi \}$, and  the function $\sigma(x, \xi)$ does not need to be in $\mathcal{D}_\MO'(\Omega) \otimes \mathcal{S}'(\I)$, in principle it is only necessary that $$\sigma_A (\cdot , \xi) \in L^2(\Ombar), \esp \text{for each} \esp \xi \in \I.$$ For this reason we will call linear operators $A$ that satisfy the condition (2) $\MO$-quantizable operators. The practical utility of this approach is reduced since it does not give enough information about the symbols to develop a symbolic calculus but, as we will show in Section 5, in some contexts this approach could be useful.
\end{rem}

Similarly, we recall the analogous notion of the $\MO^*$-quantisation.

\begin{defi}\normalfont
We say that the linear continuous operator $$A: C^\infty_{\MO^*} (\Ombar) \to \mathcal{D}_{\MO^*}'(\Omega),$$ 
belongs to the class of $\MO^*$–admissible operators if $$\sum_{\eta \in \I} v_\eta^{-1} (x) v_\eta (z) \int_{\Omega} \Tilde{K}_A (x,y) v_\eta (y) dy ,$$ is in $\mathcal{D}_{\MO^*}'(\Omega \times \Omega)$. For example, this is the case when the functions $v_\xi$ do not have any zeros in $\Omega$.
\end{defi}

So, from now on we will assume that operators $A: C^\infty_{\MO^*} (\Ombar) \to \mathcal{D}_{\MO^*}'(\Omega)$ are from the
class of $\MO$–admissible operators.

\begin{rem}
Similarly to Remark 3.13, note that the expression
$$v_\eta^{-1} (x) \int_\Omega \Tilde{K}_A (x,y) v_\eta (y) dy,$$
exists for any operator $A$ from the class of $\MO^*$–admissible operators. Moreover, it is in $\mathcal{D}_{\MO^*}'(\Omega) \otimes \mathcal{S}'(\I)$.
\end{rem}

\begin{defi}\normalfont[$\MO^*$-Symbols of operators]
The $\MO^*$-symbol of a linear continuous $\MO^*$–admissible operator $$A: C^\infty_{\MO^*} (\Ombar) \to \mathcal{D}_{\MO^*}'(\Omega),$$ is defined by $$\tau_A (x,\xi) := v_\xi^{-1} (x) \int_\Omega \Tilde{K}_A (x,y) v_\xi (y) dy.$$
\end{defi}

\begin{teo}
Let $$A: C^\infty_{\MO^*} (\Ombar) \to \mathcal{D}_{\MO^*}'(\Omega),$$ be a linear continuous $\MO^*$–admissible operator with $\MO^*$-symbol $\tau_A \in \mathcal{D}_{\MO^*}'(\Omega) \otimes \mathcal{S}'(\I)$. Then
the $\MO^*$–quantisation $$Af (x) = \sum_{\xi \in \I} \sigma_A (x , \xi) \widehat{f} (\xi) u_\xi (x),$$ is true for every $f \in C^\infty_L (\Ombar)$. The $\MO^*$-symbol of $A$ can be written as $$\tau_A (x,\xi) = v_\xi^{-1} (x) A v_\xi (x).$$
\end{teo}

\begin{rem}
Let $$A: Span\{v_\xi\} \subseteq  Dom(A)  \subseteq L^2(\Omega) \to L^2 (\Omega),$$ be a linear operator. Similarly to Remark 3.16, if there exist a measurable function $\tau_A : \Ombar \times \I \to \C$ such that \begin{align*}
    \tau_A (x,\xi) v_\xi (x) = A v_\xi(x ), \tag{3}
\end{align*} then we note that the $\MO^*$-quantisation $$Af (x) = \sum_{\xi \in \I} \tau_A (x , \xi) \widehat{f}_* (\xi) v_\xi (x), $$ is true for every $f \in Span\{v_\xi \}$, and  the function $\tau_A (x, \xi)$ does not need to be in $\mathcal{D}_{\MO^*}'(\Omega) \otimes \mathcal{S}'(\I)$, in principle it is only necessary that $$\tau_A (\cdot , \xi) \in L^2(\Ombar), \esp \text{for each} \esp \xi \in \I.$$ We will call linear operators $A$ that satisfy the condition (3) $\MO^*$-quantizable operators. 
\end{rem}

The quantizable operators whose symbol does not depend on the variable $x$ are especially important, and therefore receive a particular name.

\begin{defi}\normalfont
Let $A: Dom(A) \subset L^2(\Omega) \to L^2 (\Omega)$ be an $\MO$-quantizable operator. We will
say that $A$ is an $\MO$-Fourier multiplier if it satisfies $$\mathcal{F}_\MO (Af)(\xi) = \sigma (\xi) \widehat{f}(\xi), \esp \esp f \in Dom(A),$$ for some $\sigma: \I \to \C$. Analogously we define $\MO^*$-Fourier multipliers: Let $B:Dom(B) \subset L^2(\Omega) \to L^2 (\Omega)$ be a $\MO^*$-quantizable operator. We will say that $B$ is an 
$\MO^*$-Fourier multiplier
if it satisfies $$\mathcal{F}_{\MO^*} (Bf)(\xi) = \tau (\xi) \widehat{f}_* (\xi), \esp \esp f \in Dom(B),$$
for some $\tau: \I \to \C$.
\end{defi}

As in \cite[Proposition 3.6]{DELGADONHAnalysis}, we have the following simple relation between the symbols of
a Fourier multiplier and its adjoint.

\begin{teo}
The operator $A$ is an $\MO$-Fourier multiplier by $\sigma(\xi)$ if and only if $A^*$ is an $\MO^*$-Fourier multiplier by $\overline{\sigma(\xi)}$.
\end{teo}

Another useful result about $\MO$-Fourier multipliers is the following:

\begin{lema}
Let $A$ be an $\MO$-Fourier multiplier with symbol $\sigma(\xi)$. Then $A$ extends to a compact operator in $L^2 (\Omega)$ if and only if $$\lim_{|\xi| \to \infty} |\sigma (\xi)| = 0.$$
\end{lema}

\subsection{Difference operators and Hörmander classes} In this subsection we recall difference operators, that are instrumental in defining
symbol classes for the symbolic calculus of operators. After that we recall the definition of Hörmander classes corresponding to the present setting.

\begin{defi}\normalfont[$\MO$-stongly admissible functions]
Define $$C_b^\infty (\Omega \times \Omega) : =C^\infty (\Omega \times \Omega) \cap C(\Ombar \times \Ombar),$$ and let $q_j \in C_b^\infty (\Omega \times \Omega)$, $j=1,...,l,$ be a given family of smooth functions. We will call the collection of $q_j$'s $\MO$-strongly admissible if the following properties hold:

\begin{itemize}
    \item For every $x \in \Omega$ the multiplication by $q_j(x , \cdot)$ is a continous linear mapping on $C^\infty_\MO (\Ombar)$ for all $j=1,..,l;$
    \item $q_j (x,x) = 0$ for all $j=1,..,l$;
    \item rank$(\nabla_y q_1 (x,y),...,\nabla_y q_l (x,y))|_{y=x} = d := dim (\Omega)$;
    \item the diagonal in $\Omega \times \Omega$ is the only set when all of $q_j$'s vanish: $$\bigcap_{j=1}^{l} \{(x,y) \in \Omega \times \Omega : \esp q_j (x,y)=0\} = \{ (x,x) : \esp x \in \Omega \}.$$ 
\end{itemize}
\end{defi}
 The collection of $q_j$'s with the above properties generalises the notion of a strongly admissible collection of functions for difference operators introduced in \cite{RTWHorman} in the context of compact Lie groups. We will use the multi-index notation 
 $$q^\alpha (x,y):= q_1^{\alpha_1} (x,y) \cdot \cdot \cdot q_l^{\alpha_l} (x,y).$$
 
 \begin{defi}\normalfont[$\MO^*$-admissible operators]
 Analogously, the notion of an $\MO^*$-strongly admissible collection suitable for the conjugate problem is that of a family $\Tilde{q}_j \in C_b^\infty (\Omega \times \Omega)$, $j=1,...,l,$ satisfying the properties:
 
 \begin{itemize}
    \item For every $x \in \Omega$ the multiplication by $\Tilde{q}_j(x , \cdot)$ is a continous linear mapping on $C^\infty_{\MO^*} (\Ombar)$ for all $j=1,..,l;$
    \item $\Tilde{q}_j (x,x) = 0$ for all $j=1,..,l$;
    \item rank$(\nabla_y \Tilde{q}_1 (x,y),...,\nabla_y \Tilde{q}_l (x,y))|_{y=x} = d := dim (\Omega)$;
    \item the diagonal in $\Omega \times \Omega$ is the only set when all of $\Tilde{q}_j$'s vanish: $$\bigcap_{j=1}^{l} \{(x,y) \in \Omega \times \Omega : \esp \Tilde{q}_j (x,y)=0\} = \{ (x,x) : \esp x \in \Omega \}.$$ 
\end{itemize}
\end{defi}

We also write $$\Tilde{q}^\alpha (x,y):= \Tilde{q}_1^{\alpha_1} (x,y) \cdot \cdot \cdot \Tilde{q}_l^{\alpha_l} (x,y).$$
 From now on we will always assume that the appearing collections are strongly admissible. We now record the Taylor expansion formula with respect to a family of $q_j$'s, which follows from expansion of functions $g$ and $q^\alpha (x,\cdot)$ by the common Taylor series:
 
 \begin{pro}
 Any smooth function $g \in C^\infty (\Omega)$ can be approximated by Taylor polynomial type expansion i.e. for $x \in \Omega$, we have $$g(y) =  \sum_{|\alpha|< N} \frac{1}{\alpha!} D_y^{(\alpha)} g(y)|_{y=x} q^\alpha (x,y) + \sum_{|\alpha|=N} \frac{1}{\alpha!} q^\alpha (x,y) g_N (y),$$ in a neighbourhood of $x \in \Omega$, where $g_N  \in C^\infty (\Omega)$ and $D_y^{(\alpha)}g(y)|_{y=x} $ can be found from the recurrent formula: $D_y^{(0,..,0)} := I$ and for $\alpha \in \N_0^l,$ $$\partial_y^{\beta} g(y)|_{y=x} = \sum_{|\alpha| \leq |\beta|} \frac{1}{\alpha!}[\partial_y^{\beta} q^\alpha (x,y)]\big|_{y=x} D^{(\alpha)}_y g (y)|_{y=x},$$ where $\beta = (\beta_1 , ... , \beta_n)$. Analogously, any function $C^\infty (\Omega)$ can be approximated by Taylor polynomial type expansions corresponding to the adjoint problem, i.e. we have $$g(y) = \sum_{|\alpha| < N} \frac{1}{\alpha ! } \Tilde{D}_y^{(\alpha)} g(y) |_{y=x} \Tilde{q}^{\alpha} (x , y) + \sum_{|\alpha|=N} \frac{1}{\alpha!} \Tilde{q}^{\alpha} (x,y) g_N (y),$$ in a neighborhood of $x \in \Omega$, where $g_N (y) \in C^\infty (\Omega)$ and $\Tilde{D}_y^{(\alpha)}g(y)|_{y=x} $ can be found from the recurrent formula: $\Tilde{D}_y^{(0,..,0)} := I$ and for $\alpha \in \N_0^l,$ $$\partial_y^{\beta} g(y)|_{y=x} = \sum_{|\alpha| \leq |\beta|} \frac{1}{\alpha!}[\partial_y^{\beta} q^\alpha (x,y)]\big|_{y=x} \Tilde{D}^{(\alpha)}_y g (y)|_{y=x},$$ where $\beta = (\beta_1 , ... , \beta_n).$
 \end{pro}
 
It can be seen that operators $D^{(\alpha)}$ and $\Tilde{D}^{(\alpha)}$ are differential operators of order $|\alpha|$, and that $\partial_x^{\alpha}$ can be expressed in terms of $D^{(\alpha)}$ or $\Tilde{D}^{(\alpha)}$ as linear combination with smooth bounded coefficients. This fact will be important for Proposition 3.32. Now that we have recalled the Taylor expansion formula we  recall the definition of difference operators \cite{N-H.AnalysisRT1, N-H.AnalysisRT2}.
\begin{defi}\normalfont
Let $$A:C^\infty_\MO (\Ombar) \to \mathcal{D}_\MO' (\Omega),$$ be an $\MO$-admissible operator with the symbol $\sigma_A \in \mathcal{D}_\MO' (\Omega) \otimes \mathcal{S}'(\I)$ and with the Schwartz kernel $K_A \in \mathcal{D}_\MO' (\Omega \times \Omega)$. Then the difference operator $$\Delta_q^\alpha : \mathcal{D}_\MO' (\Omega)\otimes \mathcal{S}'(\I) \to \mathcal{D}_\MO' (\Omega)\otimes \mathcal{S}'(\I),$$ acting on $\MO$-symbols by $$\Delta_q^\alpha \sigma_A (x , \xi):= u_\xi^{-1}  (x) \int_\Omega q^\alpha (x,y) K_A (x,y) u_\xi (y) dy,$$ is well defined. Analogously, for a $\MO^*$–admissible operator $$A:C^\infty_{\MO^*} (\Ombar) \to \mathcal{D}_{\MO^*}' (\Omega),$$ with symbol $\tau_A \in \mathcal{D}_{\MO^*}'(\Omega) \otimes \mathcal{S}'(\I)$ and with the Schwartz kernel $\Tilde{K}_A \in \mathcal{D}_{\MO^*}'(\Omega \times \Omega),$ the difference operator $$\Tilde{\Delta}_q^\alpha : \mathcal{D}_{\MO^*}' (\Omega)\otimes \mathcal{S}'(\I) \to \mathcal{D}_{\MO^*}' (\Omega)\otimes \mathcal{S}'(\I),$$ acting on $\MO^*$-symbols by $$\Tilde{\Delta}_q^\alpha \tau_A (x , \xi):= v_\xi^{-1}  (x) \int_\Omega \Tilde{q}^\alpha (x,y) \Tilde{K}_A (x,y) v_\xi (y) dy,$$ is well defined.
\end{defi}
Using such difference operators and derivatives $D^{(\alpha)}$
from Proposition 3.3 it is possible to define classes of symbols.

\begin{defi}\normalfont[Symbol classes $S^m_{\rho, \delta} (\Ombar \times \I)$] The $\MO$-symbol class $S^m_{\rho, \delta} (\Ombar \times \I)$ consists of such symbols $\sigma (x , \xi)$ which are in $C^\infty_\MO (\Ombar)$ for all $\xi \in \I$, and which satisfy $$|\Delta_q^{\alpha} D_x^{(\beta)} \sigma (x, \xi)| \leq C_{\sigma, \alpha, \beta, m} \langle \xi \rangle^{m - \rho|\alpha| + \delta|\beta|} , $$
for all $x \in \Ombar$, for all $\alpha , \beta \geq 0$, and for all $\xi \in \I$. Furthermore, we define $$S^\infty_{\rho , \delta} (\Ombar \times \I) := \bigcup_{m \in \R} S^m_{\rho , \delta} (\Ombar \times \I)$$ and $$S^{- \infty} (\Ombar \times \I) := \bigcap_{m \in \R} S^m_{1 , 0} (\Ombar \times \I).$$Analogously, we define the $\MO^*$-symbol class $\Tilde{S}^m_{\rho , \delta} (\Ombar \times \I)$ as the space of those functions $\tau (x , \xi)$ which are in $C^\infty_{\MO^*} (\Ombar)$  for all $\xi \in \I$, and wich satisfy $$\big| \Tilde{\Delta}_{q}^{\alpha} \Tilde{D}_x^{(\beta)} \tau (x , \xi) \big| \leq C_{\tau, \alpha, \beta, m} \langle \xi \rangle^{m - \rho |\alpha| + \delta |\beta|},$$ for all $x \in \Ombar$ for all $\alpha , \beta \geq 0 $, and for all $\xi \in \I$. Similarly one defines the classes $\Tilde{S}^\infty_{\rho , \delta} (\Ombar \times \I)$ and $\Tilde{S}^{- \infty} (\Ombar \times \I)$.
\end{defi}

As usual, for symbols in a Hörmander class we have a symbolic calculus \cite{N-H.AnalysisRT1}. In what follows $Op_\MO (S^m_{\rho, \delta} (\Ombar \times \I))$ and $Op_{\MO^*} (\Tilde{S}^m_{\rho, \delta} (\Ombar \times \I))$ will denote the collection of linear operators with symbols in the Hörmander classes  $S^m_{\rho, \delta} (\Ombar \times \I)$ and $\Tilde{S}^m_{\rho, \delta} (\Ombar \times \I)$ respectively, defined by quantization in Theorem 3.15 and Theorem 3.20 .

\begin{lema}[Composition formula]
Let $m_1 , m_2 \in \R $ and $\rho > \delta \geq 0$. Let $A,B: C^\infty_\MO (\Ombar) \to C^\infty_\MO (\Ombar)$ be continous and linear, and assume that their $\MO$-symbols satisfy \begin{align*}
    |\Delta^\alpha_q \sigma_A (x , \xi)| \leq C_\alpha \langle \xi \rangle^{m_1 - \rho |\alpha|}, \\ |D_x^{(\beta)} \sigma_B (x , \xi)| \leq C_{\beta} \langle \xi \rangle^{m_2 + \delta |\beta|},
\end{align*}
for all $\alpha , \beta \geq 0$, uniformly in $x \in \Ombar$ and $\xi \in \I$. Then $$\sigma_{AB} (x , \xi) \sim \sum_\alpha \frac{1}{\alpha} \Delta_q^\alpha \sigma_A (x , \xi) D_x^{(\alpha)} \sigma_B (x , \xi),$$ where the asymptotic expansion means that for every $N \in \N$ we have $$\big| \sigma_{AB} (x , \xi) - \sum_{|\alpha|< N}\frac{1}{\alpha} \Delta_q^\alpha \sigma_A (x , \xi) D_x^{(\alpha)} \sigma_B (x , \xi) \big| \leq C_N \langle \xi \rangle^{m_1 + m_2 - (\rho - \delta)N}.$$
\end{lema}

\begin{lema}[Adjoint formula] Let $0 \leq \delta < \rho \leq 1$. Let $A \in Op_\MO (S^m_{\rho, \delta} (\Ombar \times \I))$. Assume that the conjugate symbol class $\Tilde{S}^m_{\rho, \delta} (\Ombar \times \I)$ is defined with strongly admissible functions $\Tilde{q}_j (x, y) := \overline{q_j (x , y)}$ which are strongly $\MO$-admissible. Then the adjoint of $A$ satisfies $A^* \in Op_{\MO^*} (\Tilde{S}^m_{\rho, \delta} (\Ombar \times \I))$, with its $\MO^*$-symbol $\tau_{A^*} \in \Tilde{S}^m_{\rho, \delta} (\Ombar \times \I)$ having the asymptotic expansion $$\tau_{A^*} (x , \xi) \sim \sum_{\alpha} \frac{1}{\alpha !} \Tilde{\Delta}_{q}^\alpha D_x^{(\alpha)} \overline{\sigma_A (x , \xi)}.$$
\end{lema}

We now show a result that will be used in the next section.
\begin{pro}Assume that the measure of $\Omega$ is finite, and that it is normalised. Then for symbols $\sigma$ in the $\MO$-symbol class $S^0_{1,0} (\Ombar \times \I)$ the series $$\sum_{\eta \in \I } \sup_{\xi \in \I}  |\widehat{\sigma} (\eta , \xi)|,$$ is convergent.
\end{pro}
\begin{proof}
Let $s_0$ be as in Assumption (B). Note that $$|\lambda_\eta^{[s_0] + 1} \widehat{\sigma} (\eta , \xi) | = \Big| \int_\Omega \sigma(x, \xi) \overline{(\lambda_\eta^{[s_0] + 1}  v_\eta (x))} dx \Big| = \Big| \int_\Omega \MO^{[s_0] + 1} \sigma (x , \xi) \overline{v_\eta (x)} dx \Big|,$$ and since $\sigma $ is in the Hörmander class $S^0_{1,0} (\Ombar \times \I)$ then $\sigma (\cdot , \xi) \in C^\infty_\MO (\Ombar)$ for each $\xi \in \I$. Hence we obtain $$|\lambda_\eta^{[s_0] + 1} \widehat{\sigma} (\eta , \xi) |  = \Big| \int_\Omega \MO^{[s_0] + 1} \sigma (x , \xi) \overline{v_\eta (x)} dx \Big| \leq \norm{\MO^{[s_0] + 1} \sigma (\cdot , \xi)}_{L^2 (\Omega)} \leq \norm{\MO^{[s_0] + 1} \sigma (\cdot , \xi)}_{L^\infty (\Omega)}.$$  Recall that, by Assumption (A), the operator $\MO$ is a differential operator with smooth bounded coefficients in $\Omega$. Then, $\MO^{[s_0] + 1}$ is a differential operator with smooth bounded coefficients in $\Omega$, what allows us to deduce that $$\sup_{\xi \in \I} || \MO^{[s_0] + 1}  \sigma (\cdot , \xi) ||_{C(\Omega)} < \infty ,$$ since $\sigma $ is in the Hörmander class $S^0_{1,0} (\Ombar \times \I)$ so, all its derivatives are uniformly bounded in $x$ and  $\xi$. This concludes the proof.
\end{proof}

\begin{rem}
The above arguments and Assumption (A) also prove that:
\begin{align*}
  \sup_{\xi \in \I} ||\widehat{\sigma} (\cdot , \xi)||_{\ell^1 (\MO)} &= \sup_{\xi \in \I} \sum_{\eta \in \I} |\widehat{\sigma} (\eta , \xi)| \cdot ||u_\eta||_{L^\infty (\Omega) } \\ &\leq C_b \sup_{\xi \in \I} \sum_{\eta \in \I} |\widehat{\sigma} (\eta , \xi)| \langle \eta \rangle^{\mu_0} \\ &\leq C_b \Big( \sum_{\eta \in \I} \langle \eta \rangle^{-2s_0} \Big)^{1/2} \sup_{\xi \in \I} ||\sigma (\cdot , \xi )||_{\mathcal{H}^{\mu_0 + s_0}_\MO (\Omega)} \\ & \leq C \sup_{\xi \in \I} ||\MO^{\frac{\mu_0 + s_0}{m}} \sigma (\cdot , \xi )||_{L^2 (\Omega)} \\ &\leq C \sup_{\xi \in \I} ||\MO^{\frac{\mu_0 + s_0}{m}} \sigma (\cdot , \xi )||_{L^\infty(\Omega)} < \infty,
\end{align*}
the last quantity being finite in view of $\MO^{k}$ being a differential operator with smooth coefficients for any $k$, and by interpolation.
\end{rem}

In view of the correspondence between quantizable linear operators and symbols, from now on we will change our perspective and think of quantizable operators as linear operators associated to given symbols.

\begin{defi}\normalfont[Pseudo-differential operators] Let $\sigma : \Ombar \times \I \to \C$ be a measurable function such that $$\sigma (\cdot , \xi) \in L^2 (\Omega), \esp \esp \text{for all} \esp \xi \in \I.$$ Then one defines its associated $\MO$-pseudo-differential operator as the linear operator acting (initially) on $Span\{u_\xi\}$ by the formula $$T_\sigma f (x) = \sum_{\xi \in \I} \sigma (x, \xi) \widehat{f}(\xi) u_\xi (x).$$The function $\sigma (x,\xi)$ is called the symbol of the operator. Analogously, given a measurable function $\tau (x , \xi)$ such that $$\tau (\cdot , \xi) \in L^2 (\Omega), \esp \esp \text{for all} \esp \xi \in \I,$$ one defines its associated $\MO^*$-pseudo-differential operator as the linear operator acting (initially) on $Span\{v_\xi\}$ by the formula $$T_{\tau} f (x) = \sum_{\xi \in \I} \tau (x, \xi) \widehat{f}_*(\xi) v_\xi (x).$$
\end{defi}

\section{{Compact operators}}
 
In this section we provide a necessary and sufficient condition for compactness of pseudo-differential operators with $\MO$-symbols in the Hörmander class $S^m_{1,0} (\Ombar \times \I)$. For this purpose we enunciate the version of Gohberg's Lemma corresponding to the present framework. A proof of this theorem will be discussed in Section 6. In what follows, for $E$ and $F$ normed spaces, $\mathfrak{K} (E,F)$ denotes the collection of compact operators in $\mathcal{L} (E,F)$.
\begin{teo}[Gohberg's Lemma] Assume that $\Omega$ has finite measure $1$. Let $T_\sigma$ be a pseudo-differential operator with $\MO$-symbol $\sigma \in S^0_{1,0} (\Ombar \times \I)$. Then $\norm{T_\sigma - K}_{\mathcal{L} (L^2 (\Omega))} \geq d_\sigma $  for all compact operator $K \in \mathfrak{K}(L^2(\Omega))$, where    
\begin{align*}
     d_\sigma:= \limsup_{|\xi| \to \infty} \{\sup_{x \in \Omega} |\sigma (x , \xi)|\}.
\end{align*}
\end{teo}

The original statement of this theorem can be found in \cite{gohberg}. A toroidal  version of this theorem can be found in \cite{Molahajloo2010}. For the version of Gohberg's Lemma on general compact Lie groups see \cite{Dasgupta2016}. The proof of Theorem 4.1 will be given in Section 6.

\begin{teo}
Assume that $\Omega$ has finite measure $1$. Let $T_\sigma$ be a pseudo-differential operator with $\MO$-symbol $\sigma \in S^0_{1,0} (\Ombar \times \I)$. Then $T_\sigma$ extends to a compact  operator in $L^2 (\Omega)$ if and only if

\begin{align*}
     d_\sigma:= \limsup_{|\xi| \to \infty} \{ \sup_{x \in \Omega} |\sigma (x , \xi)|\} = 0.
\end{align*}
\end{teo}

\begin{proof}
Assume that $d_\sigma =0$ and let $f \in C_\MO^\infty (\Ombar)$. For all $x \in \Ombar$ we have

\begin{align*}
    (T_\sigma f) (x) &= \sum_{\xi \in \I} \sigma(x , \xi) \widehat{f}(\xi) u_\xi (x) \\ &= \sum_{\xi \in \I} \Big( \sum_{\eta \in \I} \widehat{\sigma} (\eta , \xi) u_\eta(x)\Big) \widehat{f} (\xi) u_\xi (x) \\
    &= \sum_{\eta \in \I} u_\eta (x) \Big( \sum_{\xi \in \I} \widehat{\sigma} (\eta , \xi) \widehat{f} (\xi) u_\xi (x) \Big)  \\
    &= \sum_{\eta \in \I}  u_\eta (x) (T_{\widehat{\sigma}_\eta} f) (x).
\end{align*}
Here $\widehat{\sigma}_\eta (\xi) := \widehat{\sigma}(\eta, \xi)$ and the change in the order of summation is justified by Fubini–Tonelli's theorem since 

\begin{align*}
    \sum_{\xi \in \I} \sum_{\eta \in \I} |\widehat{\sigma} (\eta, \xi)| |\widehat{f}(\xi)|& \norm{u_\xi}_{L^\infty (\Ombar)} \norm{u_\eta}_{L^\infty (\Ombar)}\\ &= \sum_{\xi \in \I} \norm{\widehat{\sigma}(\cdot , \xi)}_{\ell^1 (\MO)}  \cdot |\widehat{f}(\xi) | \norm{u_\xi}_{L^\infty (\Ombar)}\\
    & \leq \sup_{\xi \in \I} \norm{\widehat{\sigma}(\cdot , \xi)}_{\ell^1 (\MO)}\cdot \sum_{\xi \in \I}  |\widehat{f}(\xi) | \norm{u_\xi}_{L^\infty (\Ombar)} < \infty .
\end{align*}

\noindent By defining the operator $(A_\eta f) (x) := u_\eta (x) f (x)$, a multiplication operator, we have

$$(T_\sigma f) (x) = \sum_{\eta \in \I}  (A_\eta  T_{\widehat{\sigma}_\eta} f) (x),$$

\noindent and clearly $A_ \eta \in \mathcal{L} (L^2(\Omega))$ since $$|| A_\eta f ||_{L^2 (\Omega)} \leq ||u_\eta||_{L^\infty (\Omega)} ||f||_{L^2 (\Omega)} \leq C_b \langle \eta \rangle^{\mu_0} ||f||_{L^2 (\Omega)} \esp \esp \text{for each} \esp \esp \eta \in \I.$$ Now, for each $\eta \in \I$, the operator $T_{\widehat{\sigma}_\eta}$ is a Fourier multiplier. Moreover, since a  pseudo-differential operator with symbol $\sigma(\xi)$ depending just on the Fourier variable extend to a compact operator in $L^2 (\Omega)$ if and only if 

    $$\lim_{|\xi| \to \infty} |\sigma (\xi)|=0,$$

\noindent and for each $\eta \in \I$ we have that

\begin{align*}
    \lim_{|\xi| \to \infty} |\widehat{\sigma} (\eta,\xi)| &= \lim_{|\xi| \to \infty} \big| \int_{\Omega} \sigma (x , \xi) \overline{u_\eta (x)} dx \big| \\& \leq \lim_{|\xi| \to \infty} ||\sigma (\cdot , \xi)||_{L^2 (\Omega)} \\ &\leq \lim_{|\xi| \to \infty} \{ \sup_{x \in \Omega} |\sigma (x , \xi)| \} \\ & \leq \limsup_{|\xi| \to \infty} \{ \sup_{x \in \Omega} |\sigma (x,\xi)| \}= 0,
\end{align*}

\noindent then each operator $T_{\widehat{\sigma}_\eta}$ is a compact operator. As a consequence each  $A_\eta  T_{\widehat{\sigma}_\eta}$ is compact and for all $N \in \N$ , the operator
$$ \sum_{|\eta| \leq N} A_\eta  T_{\widehat{\sigma}_\eta},$$
is also compact since the set of  compact operators $\mathfrak{K} (L^2 (\Omega))$ form a two sided ideal in $\mathcal{L} (L^2 (\Omega))$ (see \cite{vitali}, Proposition 4.3.4) and this ideal of compact operators is a closed subset of $\mathcal{L} (L^2(\Omega))$ in the operator norm topology. For this reason, if the series

\begin{align*}
     \sum_{\eta \in \I}  A_\eta T_{\widehat{\sigma}_\eta},
\end{align*}

\noindent converges in the operator norm topology, then 

\begin{align*}
    T_\sigma = \lim_{|N| \to \infty} \sum_{|\eta| \leq N} A_\eta T_{\widehat{\sigma}_\eta},
\end{align*}

\noindent is compact as  it is the limit of a sequence of compact operators. 
We have already seen in Remark 3.33 that if $\sigma \in S^0_{1,0} (\Ombar \times \I)$ then 

$$\sum_{\eta \in \I} \norm{A_\eta  T_{\widehat{\sigma}_\eta}}_{\mathcal{L}(L^2(\Omega))} \leq  \sum_{\eta \in \I} C_b \langle \eta \rangle^{\mu_0} \norm{T_{\widehat{\sigma}_\eta}}_{\mathcal{L}(L^2(\Omega))} \leq C_b \frac{K_1}{k_1} \sum_{\eta \in \I} \langle \eta \rangle^{\mu_0} \sup_{\xi \in \I} |\widehat{\sigma} (\eta , \xi)|,$$
where $k_1 , K_1$ are as in Lemma 3.8. The above sum converges since $$\sum_{\eta \in \I} \langle \eta \rangle^{\mu_0} \sup_{\xi \in \I} |\widehat{\sigma} (\eta , \xi)| = \sum_{\eta \in \I} \langle \eta \rangle^{-s_0} \sup_{\xi \in \I} \langle \eta \rangle^{\mu_0 + s_0}  |\widehat{\sigma} (\eta , \xi)| \leq \sum_{\eta \in \I} \langle \eta \rangle^{-s_0} \sup_{\xi \in \I} ||\sigma (\cdot , \xi)||_{\mathcal{H}_{\MO}^{s_0 + \mu_0}} < \infty.$$In summary, $T_\sigma$ is a compact operator. Now, assume that $d_\sigma \neq 0$. We need only to show that $T_\sigma$ is not compact on $L^2 (\Omega)$. Suppose that $T_\sigma$ is compact. If we set $T_\sigma = K$ in Theorem 4.1 then it contradicts our assumption that $d_\sigma \neq 0$.
\end{proof}
 Analogously, with the same scheme of proof one can prove the following theorem:

\begin{teo}
Let $T_\tau$ be a pseudo-differential operator with $\MO^*$-symbol $\tau \in \Tilde{S}^0_{1,0} (\Ombar \times \I)$. Then $T_\tau$ extend to a compact  operator in $L^2 (\Omega)$ if and only if

\begin{align*}
     d_\tau:= \limsup_{|\xi| \to \infty} \{ \sup_{x \in \Omega} |\tau (x , \xi)|\} = 0.
\end{align*}
\end{teo}

\section{{Gershgorin theory}}

In this section, under certain conditions, we will provide spectrum localisation of pseudo-differential operators in the context of the non-harmonic analysis of boundary value problems. Most of this section consists in the application of several well known results about infinite matrix theory. For this reason we will begin recalling the theorems about infinite matrices that we will use later.
In what follows for a linear operator $T :Dom(T) \subseteq E \to E$ the resolvent set of $T$ will be denoted by $$Res(T):=\{\lambda\in \mathbb{C}: (T-\lambda I)^{-1}\in \mathcal{L}(E) \},$$ and the spectrum by  $Spec(T):=\mathbb{C}\setminus Res(T)$.

\subsection{Infinite Matrices}
\begin{defi}\normalfont
Given an infinite index set $\I$, an infinite matrix indexed by $\I$ is a function $M: \I \times \I \to \C$ with matrix entries defined by $M_{\xi \eta} := M(\xi,\eta)$. If $M$ is an infinite matrix and $\varphi$ an infinite vector (or a function from $\I$ to $\C$) then the product of the vector $\varphi$ an the matrix $M$ is defined as 

\begin{align*}
    M\varphi (\xi) := \sum_{\eta \in \I} M_{\xi \eta} \varphi (\eta).
\end{align*}

For infinite matrices $P$ and $Q$ their product is defined as the infinite matrix with entries

\begin{align*}
    PQ_{\xi \eta} := \sum_{\gamma \in \I} P_{\xi \gamma} Q_{\gamma \eta},
\end{align*}

and as usual, the adjoint of the infinite matrix $M$ is the infinite matrix $M^*$ with entries

\begin{align*}
    M^*_{\xi \eta} := \overline{(M_{\eta \xi})}.
\end{align*}

\end{defi}

It is easy to see that, with the above definition, for any pair of infinite vectors (functions $\varphi_1, \varphi_2 : \I \to \C$) and complex numbers $\lambda_1 , \lambda_2$ one has

\begin{align*}
    M(\lambda_1 \varphi_1 + \lambda_2 \varphi_2)  = \lambda_1 M \varphi_1  + \lambda_2 M \varphi_2,
\end{align*}
so it is reasonable to think that an infinite matrix $ M $ can define a linear operator on some sequence space. However, not all infinite matrices define linear operators, and some conditions must be imposed on the matrix to be sufficiently well behaved. In this case we are interested in linear operators on $\ell^2 (\MO)$. Fortunately, infinite matrices that define linear operators in $ \ell^2 (\MO) $ are closely related to infinite matrices acting on $\ell^2(\N)$, (and then with matrices acting on $\ell^2 (\I)$) which have already been studied, and many results have been obtained. We state the most relevant for our work below. The following statement can be found in \cite{Crone1971}.

\begin{lema}[Crone]
Let $M$ be an infinite matrix with rows and columns in $\ell^2(\N)$. Define the projection $$P_n(x) := \sum_{k \leq n} ( x , e_k )_{\ell^2 (\N)} e_k, $$ where $e_k (j) = \delta_{kj}$. Then $M$ defines a bounded operator in $\ell^2 (\N)$ if and only if
$$\sup_{n \in \N} \norm {P_n  M^*  M  P_n}_{\mathcal{L}(\ell^2 (\N))} < \infty.$$ When this happens we have $$\sup_{n \in \N} \norm {P_n  M^*  M  P_n}_{\mathcal{L}(\ell^2 (\N))} = \norm{M}_{\mathcal{L}(\ell^2 (\N))}^2.$$
\end{lema}
\esp

With an analogous reasoning to Crone we can prove: 

\begin{lema}
Let $M$ be an infinite matrix with rows and columns in $\ell^2(\N)$. Then $M$ defines a bounded operator in $\ell^2(\N) $ if and only if $$\sup_{n \in \N} \norm {P_n    M  P_n}_{\mathcal{L}(\ell^2 (\N))} < \infty.$$ When this happens we have  $$||M||_{\mathcal{L}(\ell^2 (\N))} = \sup_{n \in \N} \norm {P_n    M  P_n}_{\mathcal{L}(\ell^2 (\N))}.$$
\end{lema}

\begin{proof}
Suppose that $M$ is bounded. Then for every $v \in \ell^2 (\N)$ and every $n \in \N$ we have $$||P_n M P_n v||_{\ell^2 (\N)} \leq ||P_n||_{\mathcal{L}(\ell^2 (\N))} || M||_{\mathcal{L}(\ell^2 (\N))}||P_n||_{\mathcal{L}(\ell^2 (\N))} ||v||_{\ell^2 (\N)} \leq ||M||_{\mathcal{L}(\ell^2 (\N))} ||v||_{\ell^2 (\N)} $$ thus  $$\sup_{n \in \N} \norm {P_n    M  P_n}_{\mathcal{L}(\ell^2 (\N))} \leq ||M||_{\mathcal{L}(\ell^2 (M))}< \infty.$$ For the converse, let $\mathcal{P}$ be the collection of vectors in $\ell^2 (\N)$ with finitely many nonzero entries. Then, for every $ v \in \mathcal{P}$, there exits a natural number $m$ such that $P_m v = v$. For this $m$ we have $$|| P_m M v||_{\ell^2 (\N)} = || P_m M P_m v||_{\ell^2 (\N)} \leq \sup_{n \in \N} || P_n M P_n||_{\mathcal{L}(\ell^2 (\N))}||v||_{\ell^2 (\N)} ,$$ and from this $$||  M v||_{\ell^2 (\N)} = \sup_{m \in \N } ||P_m M v||_{\ell^2 (\N)} \leq \sup_{n \in \N} || P_n M P_n||_{\mathcal{L}(\ell^2 (\N))}||v||_{\ell^2 (\N)},$$  yielding $$||M||_{\mathcal{L}(\ell^2 (\N))} \leq \sup_{n \in \N} \norm {P_n    M  P_n}_{\mathcal{L}(\ell^2 (\N))}< \infty.$$ The proof is complete.
\end{proof}

\begin{rem}
 We note that what the previous theorem says is: the norm of an infinite matrix, considered as a linear operator acting on $\ell^2 (\N)$, equals the supremum of the operator induced norms of a sequence of finite matrices. In fact $$\norm {P_n    M  P_n}_{\mathcal{L}(\ell^2 (\N))} = ||M_{ n}||_{\mathcal{L}(\ell^2_n (\C))} , $$ where $M_{ n}$ is the finite matrix with entries $(M)_{jk}$ for $j,k \leq n$ and $\ell^p_n (\C)$ denote the normed space $\C^n$ with the $\ell^p$-norm.  
\end{rem}

The following lemmas can be found in \cite{FARID19917, SHIVAKUMAR198735, ALEKSIC2014541}. 

\begin{lema}
Let $M$ be an infinite matrix. Define two new matrices $D$ and $F$ by $$D_{jk} := \delta_{jk} M_{jk}, \esp \esp \text{and } \esp \esp F_{jk}:=(1-\delta_{jk}) M_{jk},$$ where $\delta_{jk}$ is the Kronecker delta. If the following conditions hold

\begin{enumerate}
    \item[(i)] $M_{kk} \neq 0 $ for all $k \in \N$ and $\inf_{k \in \N} |M_{kk}| > 0$,
    \item[(ii)] $I + FD^{-1}$ defines a bounded operator in $\ell^2 (\N)$ with bounded inverse, 
\end{enumerate}

then $M$ is an invertible densely defined linear operator in $\ell^2 (\N)$ with bounded inverse. If in addition $$\lim_{|k| \to \infty} |M_{kk}| =  \infty,$$ then the inverse of $M$ is a compact operator. 
\end{lema}
\begin{lema}[Farid and Lancaster]
Let $M$ be an infinite matrix, considered as a linear operator on $\ell^p (\N)$ for $1\leq p <\infty$ fixed, with columns in $\ell^1 (\N)$. Define $r_k := \sum_{j \in \N , j \neq k} |M_{jk}|$ and assume that 

\begin{enumerate}
    \item[(i)] $M_{kk} \neq 0, \esp \text{for all} \esp \esp k \in \Z$ and $|M_{kk}| \to \infty$ as $|k| \to \infty$,
    \item[(ii)]There exist $s \in [0,1)$ such that for all $k \in \N$ $$r_k = s_k |M_{kk}|, \esp \esp s_k \in [0,s],$$
    \item[(iii)] Either $FD^{-1}$ and $(I + \mu FD^{-1})^{-1}$ exist and are in $\mathcal{L}(\ell^p (\N))$ for all $\mu \in (0,1]$, or $D^{-1}F$ and $(I + \mu D^{-1} F)^{-1}$ exist and are in $\mathcal{L}(\ell^p(\N))$ for all $\mu \in (0,1]$.
\end{enumerate}

Then $M$ is a closed operator, and the spectrum $Spec(M)$ in $\ell^p (\N)$ is nonempty and consists of discrete nonzero eigenvalues, lying in the set $$\bigcup_{k \in \N} \overline{B_\C (M_{kk} , r_k)},$$ where the closed balls $\overline{B_\C (M_{kk} , r_k)}$ are called the Gershgorin discs. Furthermore, any set consisting  of  $n$ Gershgorin discs whose union is disjoint from all other Gersgorin discs intersects $Spec(M)$ in a finite set of eigenvalues of $M$, with
total algebraic multiplicity $n$.
\end{lema}

The previous Lemmas apply in $\ell^2(\MO)$ without major modifications. Next we will adapt these theorems to pseudo-differential operators in the context of the non-harmonic analysis.

\subsection{$L^2$-Boundedness and Spectrum localisation}

Consider a measurable function $\sigma : \Ombar \times \I \to \C $  such that $\sigma (\cdot , \xi) \in L^2 (\Omega) $ for each $\xi  \in \I$, and let $T_\sigma$ be its associated $\MO$-pseudo-differential operator. Then, at least formally, for $f \in C^\infty_\MO (\Ombar)$ we can write

\begin{align*}
    (T_\sigma f) (x) &= \sum_{\xi \in \I} \Big( \sum_{\eta \in \I} \widehat{\sigma} (\eta , \xi) u_\eta (x) \Big) \widehat{f} (\xi) u_\xi (x) \\
    &= \sum_{\xi \in \I}  \sum_{\eta \in \I} \widehat{\sigma} (\eta , \xi) \widehat{f} (\xi)  u_\eta (x) u_\xi (x).
\end{align*}
Recall that by Assumption (A) $$\int_\Omega |u_\eta (x) u_\xi (x)|^2 dx \leq C_b^2 \langle \xi \rangle^{2 \mu_0} \int_\Omega |u_\eta (x)|^2 dx < \infty, $$ and thus, we can decompose the function $u_\eta (x) u_\xi (x) \in L^2 (\Omega)$ in its $\MO$-Fourier series. This means that there exist coefficients $C^{\eta \xi}_\gamma \in \C$ such that $$u_\eta (x) u_\xi(x) = \sum_{\gamma \in \I} C_\gamma^{\eta \xi} u_\gamma (x).$$
From this we have
\begin{align*}
    \sum_{\xi \in \I}  \sum_{\eta \in \I} \widehat{\sigma} (\eta , \xi) \widehat{f} (\xi)  u_\eta (x) u_\xi (x) &= \sum_{\xi \in \I}  \sum_{\eta \in \I} \sum_{\gamma \in \I} \widehat{\sigma} (\eta , \xi) \widehat{f} (\xi) C^{\eta \xi}_\gamma u_\gamma (x)\\
    &=  \sum_{\gamma \in \I} \Big( \sum_{\xi \in \I} \Big(\sum_{\eta \in \I} \widehat{\sigma} (\eta , \xi)  C^{\eta \xi}_\gamma \Big) \widehat{f} (\xi) \Big) u_\gamma (x),
\end{align*}

so, the $\gamma$-th $\MO$-Fourier coefficient of $T_\sigma f$ is 

$$ \sum_{\xi \in \I} \Big(\sum_{\eta \in \I} \widehat{\sigma} (\eta , \xi)  C^{\eta \xi}_\gamma \Big) \widehat{f} (\xi), $$

which can be writen in terms of the matrix-vector product 
$$\sum_{\xi \in \I} (M_\sigma)_{\gamma \xi} \widehat{f} (\xi),$$

where
\begin{align*}
    (M_\sigma)_{\gamma \xi} :=\sum_{\eta \in \I} \widehat{\sigma} (\eta , \xi)  C^{\eta \xi}_\gamma &= \int_{\Omega} \sum_{\eta \in \I} \widehat{\sigma}(\eta , \xi) u_\eta (x) u_\xi (x)  \overline{v_\gamma (x)} dx \\ &=\int_{\Omega} \sigma(x, \xi) u_\xi(x) \overline{v_\gamma (x)} dx \\
    &= \int_{\Omega} T_\sigma u_\xi(x) \overline{v_\gamma (x)} dx \\
    &= ( T_\sigma u_\xi , v_\gamma )_{L^2 (\Omega)}.
\end{align*} 
This observation is the key fact of this section, and is the motivation for the following definition.

\begin{defi}\normalfont[Associated matrix]
Let $\sigma : \Ombar \times \I \to \C $ be a measurable function such that $\sigma (\cdot , \xi) \in L^2 (\Omega)$ for each $ \xi  \in \I$, and let $T_\sigma$ be its associated pseudo-differential operator. Then its  associated matrix $M_\sigma$ is defined as the infinite matrix with entries $$(M_\sigma)_{\gamma \xi} :=( T_\sigma u_\xi , v_\gamma )_{L^2 (\Omega)}.$$
\end{defi}

With this definition in mind, the operator $T_\sigma $ considered as acting in $L^2 (\Omega)$ can be factored through $\ell^2 (\MO)$ as the following diagram shows

\begin{center}
\begin{tikzcd}
L^2 (\Omega) \arrow{r}{T_{\sigma}} \arrow[swap]{d}{\mathcal{F}_{L}} & L^2 (\Omega)  \\%
\ell^2(\MO) \arrow{r}{M_\sigma}& \ell^2(\MO) \arrow{u}{\mathcal{F}^{-1}_{L}}
\end{tikzcd}
\end{center}
where $\mathcal{F}_\MO$ and $\mathcal{F}_\MO^{-1}$ are the $\MO$-Fourier transform and inverse $\MO$-Fourier transform defined in Section 3. These linear operators extend to unitary operators. For this reason, the operator $T_\sigma$ is bounded in $L^2 (\Omega)$ if and only if the infinite matrix $M_\sigma$ defines a bounded operator in $\ell^2 (\MO)$, and then $\norm{T_\sigma}_{\mathcal{L} (L^2(\Omega))} = \norm{M_\sigma}_{\mathcal{L} (\ell^2(\MO))}$. Also, by Lemma 3.8, the $\ell^2 (\MO)$-norm and the $\ell^2 (\I)$-norm are equivalent so, the properties of $M_\sigma$ as a linear operator on $\ell^2 (\MO)$ (boundedness, compactness, invertibility) are the same that as those of operator on $\ell^2 (\I)$. This allows us to apply Lemma 5.2 to give necessary and sufficient conditions for the $L^2$-boundedness of pseudo-differential operators. 

\begin{teo}
Let $\sigma : \Ombar \times \I \to \C $ be a measurable function such that $\sigma (\cdot , \xi) \in L^2 (\Omega)$ for each $\xi  \in \I$, and let $T_\sigma$ be its associated $\MO$-pseudo-differential operator. Let $|M|^2_{\sigma , n}$ be the finite matrix with entries 
$$(|M|^2_{\sigma , n})_{\gamma \xi} := \sum_{\zeta \in \I} \overline{( T_\sigma u_\gamma , v_\zeta )}_{L^2 (\Omega)} ( T_\sigma u_\xi , v_\zeta )_{L^2 (\Ombar)} = \overline{( \mathcal{F}_\MO T_\sigma u_\gamma , \mathcal{F}_\MO T_\sigma u_\xi )}_{\ell^2 (\I)} \esp , \esp |\gamma|,|\xi| \leq n
.$$
Then $T_\sigma$ defines a bounded operator on $L^2 (\Omega)$ if and only if the rows of the associated matrix $M_\sigma$ are in $\ell^2 (\MO)$ (equivalently in $\ell^2 (\I)$) and  $$\sup_{n \in \N} \norm{|M|^2_{\sigma , n}}_{\mathcal{L}(\ell^2_{\mu(n)} (\C))} < \infty,$$

\noindent where $\mu(n):= \#\{\xi \in \I : \esp |\xi| \leq n\}$. When this happens we have

\begin{align*}
\frac{k_1^2}{K_1^2} \sup_{n \in \N} \norm{|M|^2_{\sigma , n}}_{\mathcal{L}(\ell^2_{\mu(n)} (\C))} \leq  \norm{T_\sigma}_{\mathcal{L}(L^2 (\Omega))}^2 \leq \frac{K_1^2}{k_1^2} \sup_{n \in \N} \norm{|M|^2_{\sigma , n}}_{\mathcal{L}(\ell^2_{\mu(n)} (\C))},   
\end{align*} 
where $k_1 , K_1$ are the constants in Lemma 3.1.
\end{teo}
\begin{proof}
We just have to see that $$
    (M_{\sigma}^* M_\sigma)_{\gamma \xi} := \sum_{\zeta \in \I} \overline{( T_\sigma u_\gamma , v_\zeta )}_{L^2 (\Omega)} ( T_\sigma u_\xi , v_\zeta )_{L^2 (\Omega)} $$
and $$\norm {P_n  M_{\sigma}^*  M_\sigma  P_n}_{\mathcal{L}(\ell^2 (\I))} = \norm{|M|^2_{\sigma , n}}_{\mathcal{L} (\ell^2_{\mu(n)} (\C))}.$$
Since the $\ell^2 (\I)$-norm and the $\ell^2 (\MO)$-norm are equivalent (Lemma 3.8) the result follows as a direct application
of Lemma 5.2.
\end{proof}

\begin{rem}
When $u_\xi = v_\xi$ for all $\xi \in \I$, the $\ell^2 (\I)$-norm and the $\ell^2 (\MO)$-norm coincide, and the matrix $|M|^2_{\sigma , n}$ takes the form $$(|M|^2_{\sigma , n})_{\gamma \xi} = ( T_\sigma u_\xi , T_\sigma u_\gamma )_{L^2 (\Omega)}.$$For example this is the case when $\MO$ is self-adjoint. 
\end{rem}

\subsection{Spectrum Localisation} The purpose of this subsection is to extend  to some class of pseudo-differential operators the theorem enunciated below.

\begin{teo}[Gershgorin Circle Theorem]
Let $M$ be a $n \times n$ matrix with entries $a_{jk}$, and define $r_j :=  \sum_{ k \neq j} |a_{jk}|$. Then each eigenvalue $\lambda$ of $M$ lies in one of the disks $\overline{B_\C (a_{jj} ,  r_j)}$.
\end{teo}

This theorem can be extended to operators that act on an infinite dimensional space, particularly to infinite matrices. There is a great quantity of literature on the subject (see for example \cite{shiva} and references therein) and indeed the Gershgorin theorem gives rise to an entire theory, called the Gershgorin theory. Lemmas 5.5 and 5.6 are examples of the achievements of this theory. Next we will rewrite their statements in the setting of the pseudo-differential operators.

\begin{teo}
Let $T_\sigma$ be a pseudo-differential operator with symbol $\sigma (x,\xi)$ such that $\sigma (\cdot,\xi) \in L^2 (\Omega)$ for each $\xi \in \I$. If $\sigma$ satisfies the following three properties:

\begin{enumerate}
        \item[(i)]$\inf_{\xi \in \I}  \Big| \int_{\Omega} \sigma (x, \xi) u_\xi (x) \overline{v_\xi (x)} dx \Big| > 0,$
    \item[(ii)]$\sup_{\xi \in \I} \Big(  \Big| \int_{\Omega} \sigma (x, \xi) u_\xi (x) \overline{v_\xi (x)} dx \Big|^{-1}  \sum_{ \zeta \neq \xi} |( T_\sigma u_\zeta , v_\xi )_{L^2(\Omega)}| \Big) < 1,$ 
    \item[(iii)]$\sup_{\xi \in \I} \Big(  \Big| \int_{\Omega} \sigma (x, \xi) u_\xi (x) \overline{v_\xi (x)} dx \Big|^{-1} \sum_{ \zeta \neq \xi} |( T_\sigma u_\xi , v_\zeta )_{L^2(\Ombar)}|\Big) < 1,$ 
\end{enumerate}
then $T_\sigma$ is an invertible linear operator with bounded inverse. In particular if

$$\lim_{|\xi| \to \infty}\big| \int_{\Omega} \sigma (x, \xi) u_\xi (x) \overline{v_\xi (x)} dx \big|= \infty,$$

\noindent the inverse is a compact operator.
\end{teo}

\begin{proof}
Let $M_\sigma$ be the associated matrix of $T_\sigma$. We will show that this infinite matrix, considered as acting on $\ell^2 (\I)$, satisfies the hypothesis of Lemma 5.5. This is enough because of Proposition 3.10, and because for any infinite matrix $M$ one has $M \in \mathcal{L} (\ell^2 (\MO))$ if and only if $M \in \mathcal{L}(\ell^2 (\I))$, and for $\lambda \in \C$,  $(M- \lambda I)^{-1} \in \mathcal{L} (\ell^2 (\MO))$
if and only if $(M - \lambda I)^{-1} \in \mathcal{L}(\ell^2 (\I))$, in virtue of Lemma 3.8. First it is easy to see that (i) and (ii) in Theorem 5.3 are equivalent to (i) in Lemma 5.3. For the remaining hypothesis define  $D_\sigma$ and $F_\sigma$ as
$$(D_\sigma)_{\gamma \xi} := \delta_{\gamma \xi} (M_\sigma)_{\gamma \xi} \esp \esp \text{and } \esp \esp (F_\sigma)_{\gamma \xi}:=(1-\delta_{\gamma \xi}) (M_\sigma)_{\gamma \xi}$$ and $$(FD^{-1}_{\sigma , n})_{\gamma \xi} := (F_\sigma D_\sigma^{-1})_{\gamma \xi}, \esp \esp |\gamma|,|\xi| \leq n.$$
Then 
\begin{align*}
    \norm{F_\sigma D_\sigma^{-1}}_{\mathcal{L}(\ell^2(\I))} &= \norm{I - (I + FD^{-1})}_{\mathcal{L}(\ell^2(\I))} \\ &= \sup_{n \in \N} ||FD^{-1}_{\sigma , n}||_{\mathcal{L}(\ell^2_{\mu(n)} (\C))}\\&\leq \sup_{n \in \N} \sqrt{ ||FD^{-1}_{\sigma , n}||_{\mathcal{L}(\ell^1_{\mu(n)} (\C))} ||FD^{-1}_{\sigma , n}||_{\mathcal{L}(\ell^\infty_{\mu(n)} (\C))}} \\ & \leq \sqrt{ \norm{F_\sigma D_\sigma^{-1}}_{\mathcal{L} (\ell^1 (\I))} \norm{ F_\sigma  D_\sigma^{-1}}_{\mathcal{L} (\ell^\infty (\I))}}. 
\end{align*}
As it is known, the operator norm of an infinite matrix acting on $\ell^1(\I)$ equals the supremum of the $\ell^1$-norms of its columns, and the operator norm on $\ell^\infty (\I)$ equals the supremum of the $\ell^1$-norms of its rows. Note that the entries of $F_\sigma D_{\sigma}^{-1}$ are \[(F_\sigma D_{\sigma}^{-1})_{\gamma \xi} = \begin{cases}
0, & \text{if } \esp \gamma = \xi, \\
\big( \int_{\Omega} \sigma (x, \xi) u_{\xi} (x) \overline{v_\xi (x)} dx \big)^{-1} (T_\sigma u_\xi , v_{\gamma})_{L^2 (\Omega)}, & \text{if} \esp \gamma \neq \xi ,

\end{cases}\]and from this we get $$\norm{F_\sigma D_\sigma^{-1}}_{\mathcal{L} (\ell^\infty (\I))} = a_1 \esp \esp \text{and} \esp \esp \norm{F_\sigma D_\sigma^{-1}}_{\mathcal{L} (\ell^1 (\I))} = a_2,$$
 \noindent where

\begin{enumerate}
    \item[] $a_1= \sup_{\xi \in \I}  \Big| \int_{\Omega} \sigma (x, \xi) u_\xi (x) \overline{v_\xi (x)} dx \Big|^{-1}  \sum_{ \zeta \neq \xi} |( T_\sigma u_\zeta , v_\xi )_{L^2(\Omega)}|$,
    \item[]$ a_2 =  \sup_{\xi \in \I}  \Big| \int_{\Omega} \sigma (x, \xi) u_\xi (x) \overline{v_\xi (x)} dx \Big|^{-1} \sum_{ \zeta \neq \xi} |( T_\sigma u_\xi , v_\zeta )_{L^2(\Omega)}|$,
\end{enumerate}
so $\norm{FD^{-1}}_{\mathcal{L}(\ell^2(\I))} < 1$ and by Lemma 2.1 in \cite{conwayfa} the operator defined by $I + {F_\sigma}{D_\sigma}^{-1}$ is invertible with bounded inverse in $\ell^2 (\I)$, consequently in $\ell^2 (\MO)$. For this reason the operator $$M_\sigma = D_\sigma + F_\sigma = (I + F_\sigma D_\sigma^{-1} )D_\sigma,$$ is invertible with bounded inverse $$D_\sigma^{-1}(I + F_\sigma D_\sigma^{-1} )^{-1},$$ wich is compact if $\sigma(x, \xi)$ satisfy $$\lim_{|\xi| \to \infty} (M_\sigma)_{\xi \xi} = \lim_{|\xi| \to \infty}\int_{\Omega}\sigma (x, \xi) u_\xi (x) \overline{v_\xi (x)} dx = \infty.$$ This completes the proof.
\end{proof}
\begin{coro}
Let $\lambda$ be a complex number and define $\sigma_\lambda (x , \xi) := \sigma(x , \xi) - \lambda$. If $\sigma_\lambda$ satisfies the hypothesis of Theorem 5.11 then $\lambda \in Res(T_\sigma)$.
\end{coro}
As an immediate consequence of Lemma 5.6 we have:
\begin{teo}
Let $\sigma : \Ombar \times \I \to \C$ be a measurable function such that $\sigma (\cdot , \xi) \in L^2 (\Omega)$ for each $\xi \in \I$, and let $T_\sigma$ be its associated pseudo-differential operator. Let $M_\sigma$ be the associated matrix. Assume that

\begin{enumerate}
    \item[(i)] $\int_{\Omega} \sigma (x, \xi) u_\xi (x) \overline{v_\xi (x)} dx \neq 0$ for all $\xi \in \I$, 
    \item[(ii)]$\lim_{|\xi| \to \infty}\big| \int_{\Omega} \sigma (x, \xi) u_\xi (x) \overline{v_\xi (x)} dx \big| = \infty$,
    \item[(iii)] Rows of $M_\sigma$ are in $\ell^2 (\I)$ and the columns are in $\ell^1 (\I)$,
    \item[(iv)]$\sup_{\xi \in \I} \Big( \Big| \int_{\Omega} \sigma (x, \xi) u_\xi (x) \overline{v_\xi (x)} dx \Big|^{-1} \sum_{ \zeta \neq \xi} |( T_\sigma u_\xi , v_\zeta )_{L^2(\Ombar)}| \Big) < 1$. 
\end{enumerate}
Then $T_\sigma$ is a closed operator and the spectrum $Spec(T_\sigma)$ is nonempty and consists of discrete nonzero eigenvalues, lying in the set $$\bigcup_{\xi \in \I} \overline{B_\C (a_{\xi} , r_\xi)},$$
where $$a_{\xi} = \int_\Omega \sigma (x, \xi) u_\xi (x) \overline{v_\xi (x)} dx \esp \esp \text{and} \esp \esp r_\xi :=\sum_{ \zeta \neq \xi} |( T_\sigma u_\xi , v_\zeta )_{L^2(\Omega)}|. $$
Furthermore, any set of $n$ Gershgorin discs whose union is disjoint from all other Gersgorin discs intersects $Spec(T_\sigma)$ in a finite set of eigenvalues of $T_\sigma$ with
total algebraic multiplicity $n$.
\end{teo}

\begin{rem}
All the analysis made in this section can be done analogously for the $\MO^*$-case, using the $\MO^*$-Fourier transform, and the infinite matrix associated to a $\MO^*$-pseudo-differential operator with symbol $\tau (x , \xi)$.  
\end{rem}

\subsection{Examples} We can use Theorem 5.13 to localise the spectrum of  operators with $\MO$-symbols of the form $\alpha(\xi) + V(x)$, in the context  some of the examples presented in Section 2.

\begin{enumerate}
    \item[(i).]  In the context of Example 2.1, consider functions $\alpha : \Z^d \to \C$  and $V \in \mathcal{F}_{\T^d}^{-1} ( \ell^1 (\Z^d))$ such that $$\alpha (\xi) \neq - \int_{\Td} V(x) dx , \esp \esp \text{for all }\esp \esp \xi \in \Z^d.$$ The associated matrix to the symbol $\sigma (x , \xi) = \alpha (\xi) + V (x)$ has entries 
\[ (M_\sigma)_{\gamma \xi} =(T_\sigma e^{i x \cdot \xi} , e^{i x \cdot \gamma}) =\widehat{\sigma} (\gamma-\xi,\xi) = 
\begin{cases} 
      \widehat{V}(\gamma-\xi), & \gamma \neq \xi, \\
      \alpha(\xi) + \int_{\Td} V(x) dx ,& \gamma=\xi,
   \end{cases}
\] and then the hypotheses that the symbol must satisfy in order to apply the Theorem 5.13 are:

\begin{enumerate}
    \item[(i)] $\alpha(\xi) + \int_{\Td} V(x) dx \neq 0$ for all $\xi \in \I$, 
    \item[(ii)]$\lim_{|\xi| \to \infty}|\alpha(\xi)| = \infty$,
    \item[(iii)] $V \in \mathcal{F}^{-1}_{\Td} (\ell^1 (\Z^d))$,
    \item[(iv)]$\sum_{ \zeta \neq \xi} |( T_\sigma e^{i x \cdot \xi }, e^{i x \cdot \zeta})_{L^2(\Td)}| = ||\mathcal{F}_{\Td} V||_{\ell^1 (\Z^d)} - \big| \int_{\Td} V(x) dx \big|< \big| \alpha (\xi) + \int_{\Td} V(x) dx \big|,$ for all $\xi \in \Z^d$.  
\end{enumerate}
Under this hypothesis the spectrum of the toroidal pseudo-differential operator associated to the symbol $\sigma (x , \xi) = \alpha (\xi) + V (x)$ is contained in the set $$\bigcup_{\xi \in \Z^d} \overline{B_\C (a_{\xi} , r)},$$
where $$a_{\xi} =\alpha (\xi) +  \int_{\Td} V(x) dx \esp \esp \text{and} \esp \esp r :=||\mathcal{F}_{\Td} V ||_{\ell^1 (\Z^d)} - \big| \int_{\Td} V(x) dx \big|,  $$ as a consequence of Theorem 5.13. This shows that the spectrum of the operator is purely discreet and the eigenvalues grow as the function $\alpha(\xi)$.   
\item[(ii).] Let us take functions $\alpha: \Z^d \to \C$ and $V: [0,2\pi]^d \to \C$ such that $\alpha (\xi)$ tends to infinity and grow at most polynomially, $V \in C^{d+1} [0,2\pi]^d$ and $\alpha (\xi) \neq - \int_{[0,2\pi]^d} V(x) dx$ for all $\xi \in \Z^d$. One can see that, for symbols $\sigma (x , \xi) = \alpha (\xi) + V (x)$, the associated matrix in the contexts of Examples 2.1 and 2.2 coincide, even when the operators are different. For this reason, as before, if we have $$\sum_{\zeta \neq 0} \big| \int_{(0,2\pi)^d} V(x) e^{-i \zeta \cdot x} dx \big| <\big| \alpha(\xi) +  \int_{(0,2\pi)^d} V(x) dx \big|,  $$ for all $\xi \in \Z^d$, then in the context of Example 2.2 the $\MO$-symbol $\sigma (x , \xi) = \alpha(\xi) + V(x)$ satisfies the conditions of Theorem 5.13, thus, as in the previous example, the spectrum of the associated $\MO$-pseudo-differential operator is contained in the set $$\bigcup_{\xi \in \Z^d} \overline{B_\C (a_{\xi} , r)},$$
where $$a_{\xi} =\alpha (\xi) +  \int_{(0, 2\pi)^d} V(x) dx,  $$ and $$ r := \sum_{\zeta \neq 0 } \big|\int_{(0, 2\pi)^d} V(x) e^{-i\zeta \cdot  x} dx \big|. $$
\item[(iii).] In the context of Example 2.5, for symbols $\sigma (x , \xi) = \alpha(\xi) + V(x)$, $\alpha: \N_0 \to \R$ and $V \in C^2 [0,1]$, we have, first $$(T_\sigma u_n , v_k)_{L^2 (\Omega)} =\alpha (n) \delta_{nk} + \int_0^1 V(x) u_n (x) v_k (x) dx, $$ second \[(T_\sigma u_n , v_n)_{L^2 (\Omega)} = 
\begin{cases}
\alpha(n) + \int_0^1 2x V(x) dx, & \text{if } \esp \esp n=0, \\
\alpha(n) + \int_0^1 4 (1-x) \sin^2 (2 \pi m x) V(x) dx, & \text{if} \esp \esp n = 2m - 1,\\ \alpha (n) + \int_0^1 4x \cos^2 (2 \pi m x) V(x) dx, & \text{if} \esp \esp n= 2m,
\end{cases}
\]and third, the sum  
$$\sum_{k \neq n} |(T_\sigma u_n , v_k)_{L^2 (\Omega)}| ,$$ is equal to $$  \sum_{k \geq 1} \big| \int_0^1 4x (1-x) \sin (2 \pi kx) V(x) dx\big| + \big| \int_0^1 4x \cos (2 \pi k x) V(x) dx \big| $$ if $n=0$; to \begin{align*}
     \big| \int_0^1 2\sin (2 \pi m x) V(x) dx \big| + \sum_{k \neq m} &\big| \int_0^1 4 (1-x) \sin (2 \pi kx) \sin(2 \pi m x) V(x) dx\big|\\ &+ \sum_{k \in \N_0}  \big| \int_0^1 4 \cos (2 \pi k x) \sin (2 \pi mx) V(x) dx \big|
\end{align*}

if $n= 2m-1$; and to 
 \begin{align*}
    \big| \int_0^1 2 x \cos (2 \pi m x) V(x) dx \big| &+ \sum_{k \in \N_0}  \big| \int_0^1 4 x (1-x) \sin (2 \pi kx) \cos(2 \pi m x) V(x) dx\big|\\ &+ \sum_{k \neq m} \big| \int_0^1 4x \cos (2 \pi k x) \cos (2 \pi mx) V(x) dx \big|
\end{align*}

if $n=2m$. In any case the above quantities are equal to $$||\mathcal{F}_L V(x) u_n (x)||_{\ell^1 (\I)} - \big| \int_0^1 V(x) u_n (x) v_n (x) dx \big|,$$so if $$\big| \alpha (n) + \int_0^1 V(x) u_n (x) v_n (x) dx \big| > || \mathcal{F}_L V(x) u_n (x)||_{\ell^1 (\I)} - \big| \int_0^1 V(x) u_n (x) v_n (x) dx \big|,$$ for all $n \in \N_0$ then the spectrum of the associated $\MO$-pseudo-differential operator $T_\sigma$ is contained in the set $$\bigcup_{n \in \N_0} \overline{B_\C (a_{n} , r_n)},$$
where $$a_{n} =\alpha (n) +  \int_{(0, 1)} V(x) u_n (x) v_n (x)dx  $$ and $$ r_n :=|| \mathcal{F}_\MO V(x) u_n (x)||_{\ell^1 (\N_0)} - \big| \int_0^1 V(x) u_n (x) v_n (x) dx \big| .$$ As before, this shows that eigenvalues of $T_\sigma$ grow as $\alpha(n)$. 
\end{enumerate}

\begin{rem}
We note that, if the eigenfunctions $(w_\xi)_{\xi \in \I}$, with corresponding eigenvalues $(\chi_\xi)_{\xi \in \I}$, of the pseudo-differential operator associated with the $\MO$-symbol $\sigma(x , \xi):= \alpha (\xi) + V (x)$ form a basis in $L^2 (\Omega)$, then one can construct the solutions  to the equation 
\begin{align*}
    \frac{\partial f}{ \partial t}  + T_\sigma f =0, \esp \esp f(0,x) = f_0 = \sum_{\xi \in \I} f_\xi w_\xi (x) \in L^2 (\Omega), \tag{HE}
\end{align*}as  $$f (t,x) = \sum_{\xi \in \I} f_\xi e^{- \chi_\xi t} w_\xi (x) .$$ This in fact is true for every pseudo-differential operator, and for elliptic $\MO$-symbols in a Hörmander class it is possible to ensure smoothness of solutions. We dedicate the following subsection to prove this fact. Part of it is the adaptation of the the work of M. Pirhayati in \cite{Pirhayati2011} to the present setting.
\end{rem}

\subsection{An application to generalised heat equations}  To begin with we have the following straightforward results.

\begin{pro}
Let $0 \leq \rho \leq 1$. Let $T_\sigma$ be a pseudo-differential operator with $\MO$-symbol $\sigma  \in  S^{m}_{\rho , 0} (\Ombar \times \I) ,$ for some $m<0$. Then if $T_\sigma$ has an eigenfunction, it is in $C^\infty_\MO (\Ombar)$. 
\end{pro}

\begin{proof}
Suppose $\sigma (x,\xi) \in S^{m}_{\rho , 0} (\Ombar \times \I)$, $m<0$. Then by Corollary 14.2 in \cite{N-H.AnalysisRT1} we have  $T_\sigma^l (L^2 (\Omega)) \subset \mathcal{H}_\MO^{-lm} (\Omega)$ for all $l \in \N$. This proves that if $f$ is an eigenfunction of $T_\sigma$ with corresponding eigenvalues $\lambda$ then $$f = \frac{1}{\lambda^l} T_\sigma^l f \in \mathcal{H}_\MO^{-lm} (\Omega),$$ for all $l \in \N$, thus $f \in C^\infty_\MO (\Ombar)$.  
\end{proof}
 An analogous result can be proved for some symbols in a positive Hörmander class, but we need ellipticity. First we study solutions of (HE) for quantizable operators.
 
\begin{pro}
Let $T_\sigma$ be a pseudo-differential operator with $\MO$-symbol $\sigma (x , \xi)$ such that $\sigma (\cdot , \xi) \in L^2 (\Omega)$ for every $\xi \in \I$. Suppose that the eigenfunctions of $T_\sigma$ form a Riesz basis in $\mathcal{H}_\MO^{s} (\Omega)$, and that the real parts of the correspondent eigenvalues $(\chi_\xi )_{\xi \in \I}$ are uniformly bounded from below by a constant. Then for an initial condition $f_0 \in \mathcal{H}_\MO^{s} (\Omega)$ the solution $f (t, \cdot) $ in the time $t$ of (HE) stay in $\mathcal{H}_\MO^{s} (\Omega)$ for all $t>0.$ 
\end{pro}

\begin{proof}
Just recall that, by definition of Riesz basis, there exists contants $k,K >0$ such that $$k \Big( \sum_{\xi \in \I} |f_\xi|^2  \Big) \leq ||\sum_{\xi \in \I} f_\xi w_\xi (x)||^2_{\mathcal{H}_\MO^{s} (\Omega)} \leq K \Big( \sum_{\xi \in \I} |f_\xi|^2 \Big),$$ 
for every $f \in \mathcal{H}_\MO^{s} (\Omega).$ With this  

\begin{align*}
    ||f(t,\cdot)||_{\mathcal{H}_\MO^{s} (\Omega)}^2&=  ||\sum_{\xi \in \I} e^{- \chi_\xi t} f_\xi w_\xi (x)||^2_{\mathcal{H}_\MO^{s} (\Omega)}\\ &\leq K \Big( \sum_{\xi \in \I} |e^{-\chi_\xi t} |^2 |f_\xi|^2 \Big) \\ &\leq K \sup_{\xi \in \I} |e^{-\mathfrak{Re}(\chi_\xi)t}|^2 \Big( \sum_{\xi \in \I} |f_\xi|^2 \Big) < \infty,
\end{align*}
finishing the proof.
\end{proof}

Now let us see that pseudo-differential operators with $\MO$-symbol in a a Hörmander class are closable. The following is an adaptation of the standard argument.

\begin{pro}
Let $0 \leq \delta < \rho \leq 1$, $m \in \R$. Let $T_\sigma \in Op_\MO (S^m_{\rho, \delta} (\Ombar \times \I))$. Then $T_\sigma : L^2 (\Omega) \to L^2 (\Omega)$ is closable with dense domain containing $C^\infty_\MO (\Ombar)$.
\end{pro}

\begin{proof}
Let $(\phi_k)_{k \in \N}$ be a sequence in $C^\infty_\MO (\Ombar)$ such that $\phi_k \to 0$ and $T_\sigma \phi_k \to f$ for some $f$ in $L^2 (\Omega)$ as $k \to \infty$. We only need to show that $f=0$. We have $$(T_\sigma \phi_k , \psi) = (\phi_k , T_{\sigma^*} \psi) , \esp \esp \psi \in C^\infty_{\MO^*} (\Ombar).$$
Let $k \to \infty$, then $(f , \psi) =0$ for all $\psi \in C^\infty_{\MO^*} (\Ombar)$. By the density of $C^\infty_{\MO^*} (\Ombar)$ in $L^2 (\Omega)$, it follows that $f=0$.
\end{proof}

Consider $T_\sigma : L^2 (\Omega) \to L^2 ( \Omega)$ with domain containing $C^\infty_\MO (\Ombar)$. Then by the previous result it has a closed extension. Let $T_{\sigma,0}$ be the minimal operator for $T_\sigma$, which is the smallest closed extension of $T_\sigma$. Then the domain $Dom(T_{\sigma,0})$ of $T_{\sigma,0}$ consists of all
functions $g \in L^2 (\Omega)$ for which there exists a sequence $(\phi_k)_{k \in \N}$ in $C^\infty_\MO (\Ombar)$ such
that $\phi_k \to g$ in $L^2 (\Omega)$ and $T_\sigma \phi_k \to f$ for some $f \in L^2 (\Omega)$ as $k \to \infty$.
It can be shown that $f$ does not depend on the choice of $(\phi_k)_{k \in \N}$ and $T_{\sigma , 0} g = f$. We define the linear operator $T_{\sigma , 1}$ on $L^2 (\Omega)$ with domain $Dom(T_{\sigma , 1})$ by the
following. Let $f$ and $g$ be in $L^2 (\Omega)$. Then we say that $g \in Dom(T_{\sigma , 1})$  and  $T_{\sigma , 1} g = f$
if and only if $$(g, T_{\sigma}^* \psi) = (f , \psi), \esp \esp \text{for all} \esp \esp \psi \in C^\infty_{ \MO^*} (\Ombar).$$
It can be proved that $T_{\sigma , 1}$ is a closed linear operator from $L^2 (\Omega)$ into $L^2 (\Omega)$ with
domain $Dom (T_{\sigma , 1})$ containing $C^\infty_\MO (\Ombar)$. In fact, $C^\infty_{\MO^*} (\Ombar)$ is contained in the domain $Dom(T_{\sigma , 1}^t)$ of the transpose $T_{\sigma , 1}^t$ of $T_{\sigma , 1}$. Furthermore, $T_{\sigma , 1} g = T_\sigma g$ for all $g$ in
$Dom (T_{\sigma , 1})$.

It is easy to see that $T_{\sigma , 1}$ is an extension of $T_{\sigma , 0}$. In fact $T_{\sigma , 1}$ is the largest
closed extension of $T_\sigma$ in the sense that if $B$ is any closed extension of $T_\sigma$ such that $C^\infty_{\MO^*} (\Ombar) \subseteq Dom(B^t)$, then $T_{\sigma , 1}$ is an extension of $B$. Such $T_{\sigma , 1}$ is called the maximal operator of $T_\sigma$. 

\noindent Now we recall the definition of ellipticity.

\begin{defi}\normalfont We say that $\sigma_A \in S^m_{\rho , 0 } (\Ombar \times \I)$ is elliptic if there exist constants $C_0 >0$ and $N_0 \in \N$ such that $$|\sigma_A  (x, \xi)| \geq C_0 \langle \xi \rangle^m,$$ for all $(x, \xi) \in \Ombar \times \I$ for which $\langle \xi \rangle \geq N_0$; this is equivalent to assuming that there exists $\sigma_B \in S^{-m}_{\rho , 0 } (\Ombar \times \I)$ such that $I - AB $, $I - BA $ are in $Op_\MO (S^{- \infty} (\Ombar \times \I))$.
\end{defi}

The following theorem is an analogue of Agmon-Douglis-Nirenberg in \cite{agmon}.

\begin{pro}
Let $T_\sigma$ be a pseudo-differential operator with $\MO$-symbol $\sigma \in S^m_{\rho , 0 } (\Ombar \times \I)$, $m >0$. Assume that $\sigma$ elliptic. Then there exist positive constants $C$ and $D>0$ such that $$C ||g||_{\mathcal{H}_\MO^{ m } (\Omega)} \leq ||T_\sigma g||_{ L^2 (\Omega)} + ||g||_{L^2 (\Omega)} \leq D ||g||_{\mathcal{H}_\MO^{ m } (\Omega)}.$$
\end{pro}

\begin{proof}
The inequality $$C ||g||_{\mathcal{H}_\MO^{m } (\Omega)} \leq ||T_\sigma g||_{L^2 (\Omega)} + ||g||_{L^2 (\Omega)}$$ is given by \cite[Theorem 14.3]{N-H.AnalysisRT1}. The inequality $$||T_\sigma g||_{L^2 (\Omega)} + ||g||_{L^2 (\Omega)} \leq D ||g||_{\mathcal{H}_\MO^{ m } (\Omega)},$$ is given by the boundedness of $T_\sigma : \mathcal{H}_\MO^{ m } (\Omega) \to \mathcal{H}_\MO^{0} (\Omega) = L^2 (\Omega)$, see \cite[Corollary 14.2]{N-H.AnalysisRT1}.
\end{proof}

\begin{pro}
Let $T_\sigma$ be a pseudo-differential operator with $\MO$-symbol $\sigma \in S^m_{\rho , 0 } (\Ombar \times \I)$ $m >0$, and assume it is elliptic. Then $Dom(T_{\sigma , 0}) = \mathcal{H}_\MO^{ m } (\Omega)$.
\end{pro}

\begin{proof}
Let $g \in\mathcal{H}_\MO^{ m } (\Omega)$. Then by using the density of $C^\infty_\MO (\Ombar)$ in $\mathcal{H}_\MO^{ m } (\Omega)$, there exists a sequence $(\phi_k)_{k \in \N}$ in $C^\infty_\MO (\Ombar)$ such that $\phi_k \to g$ in $\mathcal{H}_\MO^{ m } (\Omega)$ and therefore in $L^2 (\Omega)$ as $k \to \infty$. By Proposition 5.20, $(\phi_k)_{k \in \N}$ and $(T_\sigma \phi_k)_{k \in \N}$ are Cauchy sequences in $L^2 (\Omega)$. Therefore $\phi_k \to g$ and $T_\sigma \phi_k \to f$ for some $f \in L^2 (\Omega)$ as $k \to \infty$. This implies that $g \in Dom(T_{\sigma , 0})$ and $T_{\sigma , 0} g = f.$ Now assume that $g \in Dom(T_{\sigma , 0})$. Then there exists a sequence $(\phi_k)_{k \in \N}$ in $C^\infty_\MO (\Ombar)$ such that $\phi_k \to g$ in $L^2 (\Omega)$ and $T_\sigma \phi_k \to f$, for
some $f \in L^2 (\Omega)$. So, by Proposition 5.20, $(\phi_k)_{k \in \N}$ is a Cauchy sequence in $\mathcal{H}_\MO^{ m } (\Omega)$. Since $\mathcal{H}_\MO^{ m } (\Omega)$ is complete, there exists $h \in \mathcal{H}_\MO^{ m } (\Omega)$ such that $\phi_k \to h$ in $\mathcal{H}_\MO^{ m } (\Omega)$. This implies $\phi_k \to h$ in $L^2 (\Omega)$ which implies that $h = g \in \mathcal{H}_\MO^{ m } (\Omega) .$
\end{proof}

The following theorem shows that the closed extension of an elliptic pseudo-differential operator on $L^2 (\Omega)$  with $\MO$-symbol $\sigma \in S^m_{\rho , 0 } (\Ombar \times \I)$, $m >0,$ is unique, and moreover, by Proposition 5.21 its domain is $\mathcal{H}_\MO^{ m } (\Omega)$.

\begin{teo}
Let $T_\sigma$ be a pseudo-differential operator with $\MO$-symbol $\sigma \in S^m_{\rho , 0 } (\Ombar \times \I)$, $m >0$, and assume it is elliptic. Then $T_{\sigma , 0} = T_{\sigma , 1}$.
\end{teo}

\begin{proof}
Since $T_{\sigma , 1}$ is a closed extension of $T_{\sigma , 0}$, by Proposition 5.21 it is enough to
show that $Dom(T_{\sigma ,1 }) \subseteq \mathcal{H}_\MO^{ m } (\Omega)$. Let $g \in Dom(T_{\sigma ,1 })$. By ellipticity of $\sigma$, there exists $\tau \in S^{-m}_{\rho , 0 } (\Ombar \times \I)$ such that $$g = T_\tau T_\sigma g - R g , $$
where $R \in Op_\MO (S^{- \infty} (\Ombar \times \I))$ is an infinitely smoothing operator. Since $T_\sigma g = T_{\sigma , 1} g \in L^2 (\Omega)$, by \cite[Corollary 14.2]{N-H.AnalysisRT1}, it follows that $g \in \mathcal{H}_\MO^{ m } (\Omega)$, which completes the proof. 
\end{proof}

As an immediate consequence of this theorem we get:

\begin{coro}
 Let $T_\sigma$ be a pseudo-differential operator with $\MO$-symbol $\sigma \in S^m_{\rho , 0 } (\Ombar \times \I)$, $m >0$, and assume it is elliptic. Then if $T_\sigma$ has an eigenfunction, it is in $C^\infty_\MO (\Ombar)$.
\end{coro}

\begin{proof}
We just have to note that $T_\sigma f = \lambda f$ implies $f \in Dom (T_\sigma^l) \subseteq Dom(T_{\sigma , 1}^l) = \mathcal{H}_\MO^{ lm } (\Omega)$ for all $l \in \N$.
\end{proof}

And with this corollary we can provide a sufficient condition for smoothness of solutions to the equation (HE).

\begin{teo}
Let $T_\sigma$ be a pseudo-differential operator with symbol $\sigma \in S^m_{\rho , 0 } (\Ombar \times \I)$, $m >0$, and assume it is elliptic. Suppose that eigenfunctions $(w_\xi)_{\xi \in \I}$ (without loss of generality indexed by $\I$ and normalized) with corresponding eigenvalues $(\chi_\xi)_{\xi \in \I}$  form a Schauder basis of $L^2 (\Omega)$. Suppose that the real parts of eigenvalues of $T_\sigma$ grow at least as $\langle \xi \rangle^{\varepsilon}$ for some $\varepsilon >0$, and that $|\chi_\xi| \leq C \langle \xi \rangle^{\mu_1}$, for some $\mu_1 > 0$. Then the solution $f(t,x)$ in the time $t$ to the equation (HE) is in $C^\infty_\MO (\bar\Omega)$ for all $t>0$.
\end{teo}

\begin{proof}
As we said before, the solution of (HE) has the form $$f_t = \sum_{\xi \in \I} e^{- \chi_\xi t} f_\xi w_\xi (x),$$ where $f_\xi$ is the $\xi$-component of $f$ with respect to the basis $(w_\xi)_{\xi \in \I}$. Let us show that $f(t,x) \in \bigcap_{l \in \N} \mathcal{H}_\MO^{ lm } (\Omega)$. By Proposition 5.20 we have $$||w_\xi||_{\mathcal{H}_\MO^{ lm } (\Omega)} \leq C_1 (||T_\sigma^{l} w_\xi||_{L^2 (\Omega)} + ||w_\xi||_{L^2 (\Omega)}) \leq C \langle \xi \rangle^{l \mu_1}, $$ thus $$||f_t||_{\mathcal{H}_\MO^{ lm } (\Omega)} \leq \sum_{\xi \in \I} e^{- \mathfrak{Re}(\chi_\xi) t} |f_\xi| \cdot ||w_\xi||_{\mathcal{H}_\MO^{ lm } (\Omega)} \leq C \sum_{\xi \in \I} e^{- \mathfrak{Re}(\chi_\xi)t} \langle \xi \rangle^{l \mu_1} < \infty,$$  which completes the proof.
\end{proof}

\begin{coro}
 Let $T_\sigma$ be a self-adjoint elliptic $\MO$-pseudo-differential operator. Suppose that the real parts of eigenvalues $(\chi_\xi)_{\xi \in \I}$ of $T_\sigma$ grow at least linearly, and that $|\chi_\xi| \leq C \langle \xi \rangle^{\mu_1}$, for some $\mu_1 > 0$. Then the solution $f(t,x)$ in the time $t$ to the equation (HE) is in $C^\infty_\MO (\Ombar)$ for all $t>0$.
\end{coro}

\begin{rem}
Theorem 5.24 provides a sufficient condition to ensure that solutions to the equation (HE) are in $C^\infty_\MO (\Ombar)$ for all $t>0$. We want to remark that, in many cases, this implies the smothness of solutions, since the model operator is a diferential operator so $C^\infty_\MO (\Ombar) \subseteq C^\infty (\Omega)$ could be a natural assumption. We will exploit this fact in the next subsection. 
\end{rem}

\subsection{Stability of solutions}
In this subsection we use the scheme of the proof from \cite{munoz}  and \cite{cm2013} to give sufficient conditions to ensure that the solution $f(t,x)$ at the time $t$ of the pseudo-differential equation (HE) eventually becomes (and remains) a Morse function with distinct critical values for ``arbitrary" initial conditions. Until the end of the subsection all functions are assumed to be real valued. We start by recalling the concepts of Morse function and stabiliy for functions defined in a compact smooth manifold. Throughout this subsection we will use the following notation:
\begin{align*}
    L_\R^2(X):= \{f \in  L^2 (X) : \esp \esp \mathfrak{Im}(f)=0 \} \esp \esp \text{and} \esp \esp  C^\infty_\R (X):= \{f \in  C^\infty (X) : \esp \esp \mathfrak{Im}(f)=0 \}.
\end{align*}

\begin{defi}\normalfont
Let $\Omega$ be a smooth manifold. A smooth real-valued function on $\Omega$ is a Morse function if it has no degenerate critical points.
\end{defi}

\begin{defi}\normalfont
Let $\Ombar$ be a compact smooth manifold and let $f \in C^\infty (\Ombar)$. Then $f$ is said to be stable if there exist a neighbourhood $W_f$ of $f$ in the Whitney $C^\infty$ topology such that for each $f' \in W_f$ there exist diffeomorphisms $g,h$ such that the following diagram commutes 

\begin{center}
\begin{tikzcd}
\Ombar \arrow{r}{f} \arrow[swap]{d}{g} & \R \arrow{d}{h}  \\%
\Ombar \arrow{r}{f'}& \R 
\end{tikzcd}
\end{center}
\end{defi}

The corollary to the following fundamental theorem gives a simple characterization of stable functions which will be the key to what follows. See \cite[pp. 79-80]{stablemappigs}.

\begin{teo}[Stability theorem] Let $\Ombar$ be a smooth compact manifold and let $f \in C^\infty (\Ombar)$. Then $f$ is a Morse function with distinct critical values if and only if it is stable. 
\end{teo}

\begin{coro}
If $\Ombar$ is a smooth compact manifold and $f$ is a Morse function with distinct critical values, then there exists a neighborhood of $f$ in the $C^\infty$ topology such that $g$ is a Morse function with distinct critical values and the same number of critical points as $f$ for all $g$ in that neighborhood. In particular since $\Ombar$ is compact, there exist $r$ and $\varepsilon>0$ such that $g$ is a Morse function with distinct critical values and the same number of critical points as $f$ whenever $\norm{f-g}_{C^r (\Ombar)} < \varepsilon$ with $\norm{\cdot}_{C^r (\Ombar)}$ being a fixed norm for the $C^r$ topology.
\end{coro}

Now with this we can extend Lemma 2.1 in \cite{munoz} to pseudo-differential operators using the same scheme of proof that the authors in that paper. 

\begin{lema}
Let $\Ombar$ be a smooth compact manifold, and let $T_\sigma: L^2_\R (\Ombar) \to L^2_\R (\Ombar)$ be a linear operator acting on real valued functions, with the property that solutions to (HE) are in $C^\infty_\R (\Ombar)$. Suppose that the following conditions hold:

\begin{enumerate}
    \item[(i)] Eigenfunctions of $T_\sigma$ constitute a Schauder basis of $L_\R^2(\Ombar)$, and belong to $C^\infty_\R (\Ombar)$, 
        \item[(ii)] There exists $m \in \N_0$ and a basis $\mathcal{B}=\{\varphi_j\}$ of the direct sum of the first $m+1$ $\chi_j$-spaces $E_j$ $$ \Lambda_m :=\bigoplus_{0 \leq j \leq m} E_j \esp, \esp \esp dim(E_j) := d_j,$$ with the following property: the set $B$ of $l$-tuples   $(c_1,...,c_{l}) \in \R^l$, $l:= d_0 + ...+d_m$,  such that $\sum_j c_j \varphi_{j}$ is a Morse function with distinct critical values and $n$ critical points (for some $n$) is an open dense subset of $\R^{l}$. If constant functions are in some of the first $\chi_j$-spaces then the condition must hold with $\mathcal{B}$ basis of the orthogonal complement of constant functions in the direct sum of the first $m+1$ $\chi_j$-spaces,
        \item[(iii)] If the sequence $(\chi_j)_{j \in \N}$  is arranged in such a way that $j \leq k$ implies $\chi_j \leq \chi_k$, then $\chi_j$ grow at least as $j^{\varepsilon}$ for some $\varepsilon>0$, and $\chi_j>0$ for $j > m$. 
    \item[(iv)] For each $f \in C^\infty_\R (\Ombar)$ and every $r \in \N$ there exist $N,C$ such that the projection $h_j = \pi_j(f)$ of $f$ into the $j$-th eigenspace satisfies $$\norm{h_j}_{C^r (\Ombar)} \leq C (1+j^{N(r)}).$$
\end{enumerate}

Then there exist a set $S \subset L^2_\R (\Ombar)$, that is dense and open in the $L^2$ topology, such that for any initial condition $f_0 \in S$ if $f(t,x)$ is the corresponding solution to the equation $$\frac{\partial f}{ \partial t} + T_\sigma f= 0, \esp \esp f(0,x) = f_0,$$  on $\Ombar$ at time $t$, then there exist $T>0$ such that for $t \geq T$, $f(t,x)$ is a Morse function with distinct critical values on $\Ombar$ and $n$ critical points.
\end{lema}

\begin{proof}
Without loss of generality, we will consider the case when there are no constant functions in the first $\chi_j$-spaces. Let $S$ be the set of functions $f\in L^2_\R (\Ombar)$ whose projection onto the direct sum of the first $m+1$ $\chi_j$-eigenspaces is a Morse function, with distinct critical values, and $n$ critical points. Let $f \in S$. Let $P$ be the orthogonal projection into the subspace $\Lambda_m$, and $P^\bot$ the projection into its orthogonal complement $\Lambda_m^\bot$. Since the norms $$ || f||_{L^2_\R (\Ombar)} \esp \esp \text{and} \esp \esp ||Pf||_{L^2_\R (\Ombar)} + ||P^\bot f||_{L^2_\R (\Ombar)} $$ are equivalent, is clear that functions in a sufficiently small neighbourhood $U$ of $f$ in the $L^2$ topology will have their coefficients (with respect to any fixed basis of the direct sum of the first $m+1$ $\chi_j$-eigenspaces ) as close as desired to those of the projection of $f$ into the direct sum of the first $m+1$ $\chi_j$-eigenspaces. Then if we take a neighbourhood $U$ of $f \in S$ small enough, by condition $(ii)$ we have that $U \subset S$, hence $S$ is open. Now let $f \in L^2_\R (\Ombar)$, let $\pi_j$ be the projection onto the $\chi_j$-space, and let $g$ be obtained from $f$ such that $\pi_j (g) = \pi_j (f)$ for $j \geq m+1$ and $\pi_j (g)$ comes from slightly modifying the coefficients of each $\pi_j (f)$ with respect to $\mathcal{B}$ so that $\sum_{j=0}^m  \pi_j (g)$ is a Morse function with distinct critical values and $n$ critical points (this is again possible by condition $(ii)$). If the modification is slight enough, $g$ will be as close as desired to $f$ in $L^2_\R(\Ombar)$. Thus $S$ is dense. Next we check that if $f \in S$ and $f(t,x) = h_0 + h_1 + ...$ with $h_k = \pi_k (f)$, then $f = h_0 + e^{-\chi_1 t} h_1+...$ is a Morse function with distinct critical values and $n$ critical points for large $t$. By Corollary 5.30 it is enough to prove that for each $r$ $$\norm{f(t,\cdot) - H_m }_{C^r (\Ombar)} \to 0 \esp \esp \text{as} \esp \esp t \to \infty \esp \esp \text{where} \esp \esp H_m :=\sum_{j=0}^m e^{-\chi_j t} h_j,$$ and $m$ is as in the hypothesis $(ii)$. For fixed $t$ one has

\begin{align*}
    \norm{f(t,\cdot) - H_m}_{C^r (\Ombar)} &= \norm{e^{- \chi_{m+1} t} h_{m+1} + e^{-\chi_{m+2} t} h_{m+2} + ...}_{C^r (\Ombar)} \\ &= e^{-\chi_{m+1} t} \norm{\sum_{j=m+1}^{\infty} e^{(\chi_{m+1} - \chi_j)t} h_j}_{C^r (\Ombar)} \\ & \leq e^{-\chi_{m+1}t}  \sum_{j=m+1}^{\infty} e^{(\chi_{m+1} - \chi_j)t} \norm{h_j}_{C^r (\Ombar)} \\ & \leq C e^{-\chi_{m+1}t}  \sum_{j=m+1}^{\infty} e^{(\chi_{m+1} - \chi_j)t} (1+j^{N(r)}).
\end{align*}

In virtue of (3) $\chi_j$ grow at least as $j^{\varepsilon}$ for some $\varepsilon>0$, so the series $$ C e^{-\chi_{m+1}t}  \sum_{j=m+1}^{\infty} e^{(\chi_{m+1} - \chi_j)t} (1+j^{N(r)}),$$ is clearly convergent and a decreasing function of $t$. Since the first factor tends to zero as $t \to \infty$ the proof is complete. 
\end{proof}

\begin{rem}
The above lemma is about smooth functions on compact smooth manifolds, with (smooth) boundary or without any boundary (closed manifolds). However, we have used the notation $\Ombar$ here to suggest that it can be applied in the setting of the non-harmonic analysis, but for this it is necessary to have the smoothness of the boundary $\partial \Omega$, and the condition $C^\infty_\MO (\Ombar) \subseteq C^\infty (\Omega)$.   
\end{rem}

The motivation for Lemma 5.31 is the fact that the solutions of the heat equation in a wide class of manifolds become minimal Morse functions with distinct critical values. This is Lemma 2.1 in \cite{munoz} where in particular the cases $\mathbb{RP}^d$ and $\mathbb{CP}^d$ were treated. See \cite{cm2013} for the cases $\mathbb{S}^d$ and $\Td$. In our setting the case $\Td$ correspond to the periodic boundary value problem associated to the Laplacian. 

We note that, in order to apply Lemma 5.31, it is necessary to ensure three things: first, eigenvalues of $T_\sigma$ grow at a reasonable rate, second, Morse functions are dense in the first non-trivial eigenspaces, and third, the $C^r$-norm of the projection of a function in each $\chi_j$-eigenspace is bounded by some polynomial in $j$. The Laplacian is particulary nice because it is self-adjoint and its eigenvalues are well known in many cases. Moreover, on some manifolds as in those examples given in Section 2 there exist enough informaton about the basis of the first non-trivial eigenspace, and about the basis of each eigenspace. However, for more general operators it is a non-trivial problem to obtain information about its eigenfunctions, but one can use the spectrum localisation achieved by Theorem 5.13 to give at least one of the necessary conditions, in some cases. Now we give some examples where Lemma 5.31 can be applied:

\begin{exa}
Eigenfunctions of the model operator in Examples 2.1 and 2.7 coincide with the eigenfunctions of Laplace operator in $\Td$ and $\mathbb{S}^2$, respectively. In \cite{cm2013} the authors show that Lemma 5.31 applies for the heat equation and its solutions become and remain as minimal Morse functions, but much more can be said. Since we know how eigenfunctions of Fourier multipliers should be, then we can check (more or less easily in some cases)  if Lemma 5.31 applies for the equation determined by a given Fourier multiplier. For example, let $\sigma: \Z^d \to \R$ be a positive function that grows at least linearly and takes its minimum value only in integer vectors $\xi \in \Z^d $ of the form $\xi = k e_j$. Then, by the same arguments as in \cite{cm2013}, the solution $f_t$ at the time $t$ for the Cauchy problem $$ \frac{\partial f}{\partial t} + T_\sigma f = 0, \esp \esp \esp f(0,x) = f_0 ,$$ on the torus becomes and remains a Morse function with distinct critical values in view of Lemma 5.31.
\end{exa}

\begin{exa}
Consider the equation $$ \frac{\partial f}{ \partial t}+ \MO_k f = 0 , \esp \esp f(0,x) = f_0 ,$$ where $$\MO_k:= \frac{\partial^{2k}}{\partial x^{2k}} + \frac{\partial^{2k}}{\partial y^{2k}},$$in the Möbius strip with the Dirichlet boundary conditions (Example 2.9). A function in the first non trivial eigenspace associated to the operator $\MO_k$ has the form  $$h(x,y)= c \sin{(x/2)} \sin (2y),$$ which is Morse for $c \neq 0$. Each eigenspace is one-dimensional and the $C^{r}$ norm of the projection of a smooth function onto the $j$-th eigenspace is bounded by a constant times $1+j^r$. 
\end{exa}

\begin{rem}
We have used the fact that functions in Examples $5.33$ and $5.34$ can be identified with functions on a compact smooth manifold, but Lemma 5.31 works in a wider class of domains. To see this consider the Example 2.3. Since the domain in consideration is $(0 , 2 \pi)^d$ at first appearance Lemma 5.31 does not apply, but one can see for the model operator that, after ordering the eigenvalues in non-decreasing order, functions in each $\chi_j$-space are very similar to the functions in the $\chi_j$-space of the model operator in Example 5.33. So, it is reasonable to think that this difficulty can be avoided. Certainly the eigenfunctions of $\MO_{h,d}$ and consequently their linear combinations can be extended to a larger domain (a ball for example) containing $(0 , 2 \pi)^d$. We can choose this domain in such way that it is a compact smooth manifold with boundary where Lemma 5.31 applies. Moreover, since critical points of a Morse function in a compact smooth manifold are finite then we can choose an extended domain where the extension of the functions have the same number of critical points as the original functions, for example in a cube of rounded corners tight to $(0,2 \pi)^d$. This observation is the motivation of the following corollary of Lemma 5.31.
\end{rem}
\begin{coro}
 Let $\Omega \subseteq \R^d$ be an open set, and let $T_\sigma: L^2_\R (\Ombar) \to L^2_\R (\Ombar)$ be a linear operator acting on real valued functions, with the property that solutions to (HE) are in $C^\infty_\R (\Ombar)$, and such that its eigenfunctions $w_j$ are smooth in $\Omega$ and form a Schauder basis of $L_\R^2 (\Omega)$. Assume that the corresponding eigenvalues grow at least linearly, and suppose that there exists a open subset $\Omega' \subset \R^d$ such that:

\begin{enumerate}
    \item[(i)] $\Omega \subset \Omega'$ and $\overline{\Omega'}$ is a compact smooth manifold with boundary.
    \item[(ii)] Each eigenfunction $w_j$ of the operator $T_\sigma$ extends to a smooth function $w_j'$ in $\overline{\Omega'}$.
    \item[(iii)] There exists $m \in \N_0$ and a basis $\mathcal{B}=\{\varphi_j\}$ of the direct sum of the first $m+1$ $\chi_j$-spaces $E_j$ $$ \Lambda_m :=\bigoplus_{0 \leq j \leq m} E_j \esp, \esp \esp dim(E_j) := d_j,$$ with the following property: the set $B$ of $l$-tuples   $(c_1,...,c_{l}) \in \R^l$, $l:= d_0 + ...+d_m$,  such that $\sum_j c_j \varphi_{j}$ extend to a Morse function in $\Ombar'$ with distinct critical values and $n$ critical points (for some $n$) is an open dense subset of $\R^{l}$. If constant functions are in some of the first $\chi_j$-spaces then the condition must hold with $\mathcal{B}$ basis of the orthogonal complement of constant functions in the direct sum of the first $m+1$ $\chi_j$-spaces,
    \item[(iv)] For each function $f$ in the $\chi_j$-eigenspace and every $r \in \N$ there exist $N,C$ such that  $$||f||_{C^r (\overline{\Omega'})} \leq C (1+j^{N(r)}).$$
    \end{enumerate}
Then there exist a set $S \subset L^2_\R (\Omega)$ that is open and dense in $L^2_\R (\Omega)$ in the $L^2$ topology such that, for any initial condition $f_0 \in S$, if $f(t,x)$ is the corresponding solution to the equation
$$ \frac{\partial f}{ \partial t}+T_\sigma f = 0, \esp \esp f(0,x) = f_0,$$  on $\Omega$ at time $t$, then there exist $T>0$ such that for $t \geq T$, $f(t,x)$ is a Morse function with distinct critical values on $\Omega$.
\end{coro}

\begin{proof}
Let $S$ be the set of functions $f \in L^2_\R (\Omega)$ whose projection onto the direct sum of the first $m+1$ $\chi_j$-spaces $E_0 \oplus .. \oplus E_m$  extend to a Morse function in $\Ombar'$. Then, as before, $S$ is dense and open in the $L^2_\R (\Omega)$ topology, and if $f_0 \in S$ then $f(t,x)$ is eventually very close to a Morse function with distinct critical values.
\end{proof}

\begin{exa}
In Example 2.3 the first eigenspace of the operator $\MO_{h,d}$ is not trivial but the gradient of every non-zero function is non-zero in $(0,2 \pi)^d$ so, we have to consider the direct sum of the first two eigenspaces, let us call them $E_0$ and $E_1$. A function in $E_0 \oplus E_1$ can be extended to any subset or $\R^d$ containing $(0,2\pi)^d$ and can be written in the form \begin{align*}
    f(x) &= h^{x/2\pi}( a_0 + \sum_{j=1}^d a_j \cos (x_j) + b_j \sin (x_j) ),\\  
\end{align*}

thus \begin{align*}
    \frac{\partial f}{\partial x_j}(x)  &=\frac{\ln (h_j)}{2\pi}f(x)  + h^{x/2 \pi} \big(   b_j \cos(x_j) - a_j \sin (x_j)\big) \\ &= \frac{\ln (h_j)}{2\pi}f(x)  + h^{x/2 \pi} A_j \sin(x_j + \phi_j),
\end{align*}

where $$A_j = \sqrt{a_j^2 + b_j^2}, \esp \esp \phi_j = atan2(b_j , -a_j) ,$$

where $atan2$ is the two argument arctangent function, defined as the angle in the Euclidean plane, given in radians, between the positive x-axis and the ray to the point $(x,y)$. By the direct calculation 
\begin{align*}
    \frac{\partial^2 f}{\partial x_k \partial x_j} (x) =& \frac{\ln(h_k) \ln(h_j)}{4 \pi^2} f(x) +  \frac{\ln(h_k)}{2 \pi} h^{x/2 \pi} \big(   b_j \cos(x_j) - a_j \sin (x_j)\big) \\ & +\frac{\ln (h_j)}{2 \pi} h^{x/ 2 \pi} (b_k \cos (x_k) - a_k \sin(x_k) ) \\ & -\delta_{jk} h^{x/2\pi} (a_j \cos (x_j) + b_j \sin (x_j)) .
\end{align*}

Now let us suppose that $x^0$ is a critical point, then $$-\frac{\ln (h_j)}{2\pi}f(x^0) = h^{x^0/2\pi} \sin (x_j^0 + \phi_j), $$
 and 
\begin{align*}
    \frac{\partial^2 f}{\partial x_k \partial x_j} (x^0) &= \frac{\ln (h_j)}{2 \pi} h^{x^0/ 2 \pi} (b_k \cos (x_k^0) - a_k \sin(x_k^0) ) \\ & -\delta_{jk} h^{x^0/2\pi} (a_j \cos (x_j^0) + b_j \sin (x_j^0)).
\end{align*}
For the case $d=2$ we obtain

\begin{align*}
    &det (\mathbb{H}_f (x^0))\\ &= \Big( \Big(1 + \frac{\ln (h_1)}{2\pi}\Big)\Big( 1 + \frac{\ln (h_2)}{2\pi} \Big) - \frac{\ln(h_1) \ln(h_2)}{4 \pi^2} \Big)A_1 A_2 \sin (x_1^0 + \phi_1) \sin (x_2^0 + \phi_2),
\end{align*}

\noindent and from this we can see that the function $f$ is a Morse function if and only if $f(x^0) = 0$. To finish we just have to note that for any given function $f$ in $E_0 \oplus E_1$ a slight modification of the coefficients $a_j,b_j$ makes $f$ a Morse function, if it is not Morse yet. In summary, Lemma 5.31 applies in this case and in conclusion, there exist a dense set $S \subset L^2_\R (\Omega)$ such that for any $f_0 \in S$ the solution $f(t,x)$ in the time $t$ to the equation $$\MO_{h,2} f + \frac{\partial f}{\partial t}=0,  \esp \esp f(0,x) = 0 $$ become and remains as a Morse function with distinct critical values.  
\end{exa}

\begin{exa}
Consider Example 2.4. Let $\sigma : \N_0 \times \N_0 \to \R$  a positive function such that $\sigma$ takes its minimum value in a single point $(n_0,m_0) \in \N_0 \times \N_0$, $(n_0,m_0) \neq (0,0)$.  Then, for $f_0$ in a dense subset of $L^2_\R (\Omega)$ solutions to the equation $$ \frac{\partial f}{ \partial t} +T_\sigma =0, \esp \esp f(0,x) = f_0,$$ become and remain Morse function with different critical values and the same number of critical points as $$\cos\big( \frac{n_0 x}{a} \big) \sin \big( \frac{m_0 y}{b} \big).$$ 
\end{exa}

\section{{Gohberg's lemma}}

This section is dedicated to the proof of Gohberg's Lemma (Theorem 4.1) in the context of the non-harmonic analysis of boundary value problems.

\begin{proof}[Proof of Gohberg's lemma:]
Our proof consists of three parts.

\textbf{First:} since $\sigma (x , \xi)$ is bounded in $\Ombar$ then for each $\xi \in \I$ we can take a $x_\xi \in\Omega$ such that the value $|\sigma(x_\xi , \xi)|$ is arbitrarily close to $||\sigma(\cdot , \xi)||_{L^\infty (\Omega)}$. Now by definition of $d_\sigma$ we can take a subcollection $\{ (x_{\xi_k} , \xi_k) \}_{k \in \N}$ of $\{ (x_\xi , \xi)\}_{\xi \in \I}$ so that $$\lim_{k \to \infty} |\sigma (x_{\xi_k} , \xi_k)| = d_\sigma.$$ By the compactness of $\Ombar$ the collection $\{x_{\xi_k}\}$ must have an acumulation point $x_0$. This implies that each neigbourhood $V$of $x_0$ contain infinitely many points of $\{x_{\xi_k}\}$. Thus there exists a subsequence $\{x_{\xi_{k_l}}\}_{l \in \N}$ of points in the set $V$ that satisfy $$\lim_{l \to \infty} |\sigma (x_{\xi_{k_l}} , \xi_{k_l})| = d_\sigma.$$ For simplicity we will rename this sequence as the original $\{x_{\xi_k}\}$. Now let $\varepsilon$ be an arbitrary positive real number. Let us take $V_\varepsilon$ and $f \in C^\infty (\Omega)$ a smooth bounded bump function so that
$$|\sigma(x , \xi_k) - \sigma(x_{\xi_k} , \xi_k ) | < \frac{\varepsilon}{3} \esp \esp \text{for} \esp \esp x, x_k \in V_\varepsilon,$$  and $$f(x) =0 \esp \esp \text{for} \esp \esp x \notin V_\varepsilon \cap \Omega.$$ If we define $$f_k (x) := f(x) u_{\xi_k} (x),$$ then  \begin{align*}
    \norm{\sigma(x_{\xi_k} , \xi_k) f_k}_{L^2(\Omega)} - &\norm{\sigma(\cdot , \xi_{k}) f_k}_{L^2(\Omega)} \\
    &\leq \norm{\sigma(x_{\xi_k} , \xi_k) f_k - \sigma(x , \xi_k) f_k}_{L^2(\Omega)} \\ &= \Big( \int_{V_\varepsilon \cap \Omega} |\sigma(x_{\xi_k} , \xi_k) - \sigma(x , \xi_k)|^2|f_k (x)|^2 dx \Big)^{1/2} \\
    &\leq \frac{\varepsilon}{3} \norm{f_k}_{L^2(\Omega)} \tag{G1}.
\end{align*}

\textbf{Second:} we assert that the sequence $\{f_k\}_{k \in \N}$ converges to zero weakly. For this we just have to see that given any $g \in C^\infty_{\MO^*} (\Ombar)$ we have that
\begin{align*}
    \int_{\Omega} f(x) g(x) u_{\xi_k} (x) dx 
\end{align*}
is the complex conjugate of the $\MO$-Fourier coefficient $\widehat{h} (\xi_k)$ of the function $h = \overline{fg} \in L^2 (\Omega)$, and obviously $| \widehat{h} (\xi_k) | \to 0$ as $k \to \infty$. Hence for sufficiently large $k$ and any compact operator $K$, we have
\begin{align*}
    \norm{Kf_k}_{L^2(\Omega)} \leq \frac{\varepsilon}{3} \norm{f_k}_{L^2(\Omega)}, \tag{G2}
\end{align*}
because compact operators map weakly convergent sequences into strongly convergent sequences.

\textbf{Third:} we have the following lemma.

 \hfill

\begin{lema}$\esp \esp \norm{\sigma(\cdot , \xi_k)f_k - T_\sigma f_k}_{L^2(\Omega)} \to 0$ as $k \to \infty$.
\end{lema}

\hfill

If we assume the lemma for a moment then, for sufficiently large $k$ \begin{align*}
    \norm{\sigma(x, \xi_k)f_k }_{L^2(\Omega)} - \norm{T_\sigma f_k}_{L^2(\Omega)} \leq \frac{\varepsilon}{3} \norm{f_k}_{L^2(\Omega)}. \tag{G3}
\end{align*}

So, by (G1), (G2) and (G3) we get for sufficiently large $k$ that

\begin{align*}
    \norm{f_k}_{L^2(\Omega)} \norm{T_\sigma - K}_{\mathcal{L}(L^2(\Omega))} &\geq \norm{(T_\sigma - K)f_k}_{L^2(\Omega)} \\ & \geq \norm{T_\sigma f_k}_{L^2(\Omega)} - \norm{K f_k}_{L^2(\Omega)} \\ & \geq \norm{T_\sigma f_k}_{L^2(\Omega)} - \frac{\varepsilon}{3} \norm{f_k}_{L^2(\Omega)} \\ & \geq \norm{\sigma(\cdot , \xi_k) f_k}_{L^2(\Omega)} - \frac{2 \varepsilon}{3} \norm{f_k}_{L^2(\Omega)} \\ & \geq (|\sigma(x_{\xi_k} , \xi_k)| - \varepsilon) \norm{f_k}_{L^2(\Omega)}.
\end{align*}

Letting $k \to \infty$ we get $$\norm{T_\sigma - K}_{\mathcal{L}(L^2 (\Omega))} \geq d_\sigma - \varepsilon.$$
 Finally, using the fact that $\varepsilon$ is an arbitrary positive number, we have $$\norm{T_\sigma - K}_{\mathcal{L}(L^2 (\Omega))} \geq d_\sigma.$$ The proof is complete. 
\end{proof}

\begin{proof}[Proof of Lemma 6.1:]
In the distribution sense we can write 
\begin{align*}
    T_\sigma f (x) &= \sum_{\xi \in \I} \sigma(x,\xi) \widehat{f} (\xi) u_\xi (x)\\ &= \sum_{\xi \in \I} \int_\Omega \sigma(x,\xi) u_\xi(x) \overline{v_\xi (y)} f(y) dy \\ &= \int_\Omega \Big( \sum_{\xi \in \I}  \sigma(x,\xi) u_\xi(x) \overline{v_\xi (y)} \Big) f(y) dy \\ &=\int_\Omega K(x,y) f(y) dy,
\end{align*}

for any $f \in C^\infty_{\MO^*} (\Ombar)$. In particular for $f_k (y) := f(y) u_{\xi_k} (y)$ we have $$T_\sigma f_k (x) = \int_{\Omega} K(x,y) f(y) u_{\xi_k} (y) dy.$$ 

By Taylor's formula (Proposition 3.27) given any $x \in \Omega$, we have $$f(y) = \sum_{|\alpha|<N} \frac{1}{\alpha!} D^{(\alpha)}_y f(y)|_{y=x} q^\alpha (x,y) + \sum_{|\alpha|=N} \frac{1}{\alpha!} q^\alpha (x,y) f_N(x),$$ in some neighborhood of $x$. Thus

\begin{align*}
    T_\sigma f_k (x) &= f(x) \int_{\Omega} K(x,y) u_{\xi_k} (y) dy \\ & \esp \esp+ \sum_{|\alpha|=1}^N \frac{1}{\alpha!} D^{(\alpha)}_y f(y)|_{y=x}\int_{\Omega} K(x,y) q^\alpha (x,y)u_{\xi_k} (y) dy \\
    &   \esp \esp +\sum_{|\alpha|=N} \int_{\Omega} K(x,y) q^\alpha (x,y) f_N(x) \\ &= \sigma(x , \xi_k) f(x) u_{\xi_k} (x) + \sum_{|\alpha|=1}^N \frac{1}{\alpha!} D^{(\alpha)}_y f(y)|_{y=x} \Delta_q^\alpha \sigma (x, \xi_k) u_{\xi_k} (x) \\
    & \esp \esp + \Delta_{q^N} \sigma (x , \xi_k) u_{\xi_k} (x).
\end{align*}

Then we have 

\begin{align*}
    T_\sigma f_k(x) - \sigma(x , \xi_k) f_k (x) =& \sum_{|\alpha|=1}^N \frac{1}{\alpha!} D^{(\alpha)}_y f(y)|_{y=x} \Delta_q^\alpha \sigma (x, \xi_k) u_{\xi_k} (x) \\ & + \Delta_{q^N} \sigma (x , \xi_k) u_{\xi_k} (x).
\end{align*}

Finally, since $f \in C^\infty_{\MO^*} (\Ombar)$ and $\sigma \in S^0_{1,0} (\Ombar \times \I)$, it is clear that there exist constants $C_1>0$ and $C_2>0$ such that $$\Big|  \sum_{|\alpha|=1}^N \frac{1}{\alpha!} D^{(\alpha)}_y f(y)|_{y=x} \Delta_q^\alpha \sigma (x, \xi_k) u_{\xi_k} (x) \Big| \leq C_1 \langle \xi_k \rangle^{-1},$$and $$ |\Delta_{q^N} \sigma (x , \xi_k) u_{\xi_k} (x)| \leq C_2 \langle \xi_k \rangle^{-N}.$$This implies $$\norm{T_\sigma f_k - \sigma(\cdot , \xi_k) f_k}_{L^2(\Omega)} \to 0,$$ as $k \to \infty$, concluding the proof.
\end{proof}

\section*{{Acknowledgments}}
The first author thanks the support of Carlos Andres Rodriguez Torijano during the development of this work.
\nocite{*}
\bibliographystyle{acm}
\bibliography{main}

\end{document}